%% file: ex_article.tex
\documentclass[onefignum,onetabnum]{siamart190516}
\usepackage{booktabs,hhline}

\input{ex_shared}
\usepackage{pgfplots}
\ifpdf
\hypersetup{
  pdftitle={},
  pdfauthor={Bolong Zhang}
}
\fi
\newcommand{\R}{\mathbb{R}}

\myexternaldocument{ex_supplement}

\pgfplotsset{compat=1.14}

\begin{document}

\maketitle

\begin{abstract}
Low-rank matrix approximation is extremely useful in the analysis of data that arises in scientific computing, engineering applications, and data science.  However, as data sizes grow, traditional low-rank matrix approximation methods, such as singular value decomposition (SVD) and column pivoting QR decomposition (CPQR), are either prohibitively expensive or cannot provide sufficiently accurate results.
A solution is to use randomized low-rank matrix approximation methods such as randomized SVD , and randomized LU decomposition on extremely large data sets.
In this paper, we focus on the randomized LU decomposition method. First, we employ a reorthogonalization procedure to perform the power iteration of the existing randomized LU algorithm  to compensate for the rounding errors caused by the power method.  Then we propose a novel randomized LU algorithm, called PowerLU, for the fixed low-rank approximation problem. PowerLU allows for an arbitrary number of passes of the input matrix, $v \geq 2$.  Recall that the existing randomized LU decomposition only allows an even number of passes.  We prove the theoretical relationship between PowerLU and the existing randomized LU. Numerical experiments show that our proposed PowerLU is generally faster than the existing randomized LU decomposition, while remaining accurate. 
We also propose a version of PowerLU, called PowerLU\_FP, for the fixed precision low-rank matrix approximation problem.  PowerLU\_FP is based on an efficient blocked adaptive rank determination Algorithm \ref{alg:adaptivewithoutupdating} proposed in this paper. 
We present numerical experiments that show that PowerLU\_FP can achieve almost the same accuracy and is faster than the randomized blocked QB algorithm by Martinsson and Voronin.   We finally propose a single-pass algorithm based on LU factorization.  Tests show that the accuracy of our single-pass algorithm is comparable with the existing single-pass algorithms. 

\end{abstract}



\section{Introduction}
In research areas such as data mining, scientific computing, and many engineering applications, the matrices encountered are extremely large.  In \cite{udell2019big}, the authors state that big data matrices are generally low-rank.  Low-rank matrix approximation is a critical technique used to analyze data in many disciplines that have extremely large matrices \cite{mahoney2011randomized}.  Applications of randomized low-rank approximation including imaging processing \cite{elad2006image}, data mining \cite{drineas2016randnla}, and machine learning \cite{dai2018fast} have already been widely explored. Mathematically low-rank approximation of a matrix can be defined as follows.

\begin{definition}
For a given matrix $\mathbf{A} \in \mathbb{R}^{m \times n}$ with rank $k$,  we seek two low-rank matrices $\mathbf{B} \in \mathbb{R}^{m \times k} $, and $\mathbf{C} \in \mathbb{R}^{k \times n}$ with given rank $k \ll min(m,n)$ such that the following norm equation (\ref{eq:lowrank}) holds:   
\begin{equation}
 \|\mathbf{A}-\mathbf{B} \mathbf{C}\|_{N} \ll 1.  
 \label{eq:lowrank}
\end{equation}
We want to approximate $\mathbf{A}$ so that $\|\mathbf{A}-\mathbf{B} \mathbf{C}\|_N$ as small as possible.  We can use any matrix norm, but it customary to use either the spectral norm, $N = 2$, or the Frobenius norm, $N = F$ in numerical linear algebra.
\label{def:lowrank}
\end{definition}
 The above definition (\ref{def:lowrank}) is called the fixed rank problem in low-rank approximation, where we know the rank, $k$, of the matrix in advance. However, in some cases,  the rank of the matrix cannot be known in advance, and instead of giving us the rank, $k$, a real-valued tolerance $\epsilon$ is provided for the norm equation (\ref{eq:lowrank}). In the other words, we want to find the minimum rank $k$ for $\mathbf{B}$ and $\mathbf{C}$, such that equation (\ref{eq:lowrank}) holds for the given tolerance, $\epsilon$. This is another type of low-rank approximation problem, called fixed precision.

The singular value decomposition (SVD) is a common deterministic method to compute the low-rank matrix approximation.  The famous theorem by Eckart and Young \cite{eckart1936approximation} states that the SVD can achieve optimal low-rank approximation results in both the spectral and Frobenius norms.  Suppose the SVD decomposition for matrix $\mathbf{A} \in \R^{m\times n}, m \geq n$ is the following: 
\begin{equation}
\label{eq:svd}
    \mathbf{A} = \mathbf{U}\mathbf{\Sigma}\mathbf{V}^{T}, 
\end{equation}
where $\mathbf{U} \in \mathbb{R} ^ {m \times n}$ has orthonormal columns,  $\mathbf{V} \in \mathbb{R} ^ {n \times n}$ is an orthogonal matrix,  and  $\mathbf{\Sigma} \in \mathbb{R} ^ {n \times n}$ is a diagonal matrix whose diagonal entries are the singular values of $\mathbf{A}$ in descending order. Then the optimal solution to the fixed rank problem (\ref{eq:lowrank}) for a given rank $k$ is $\mathbf{A}_{k} = \mathbf{B} \mathbf{C}$, where we set $\mathbf{B} =\mathbf{U}(:, 1:k)\mathbf{\Sigma}(1:k, 1:k)$ and $\mathbf{C} = \mathbf{V}(:, 1:k)^{T}$\footnote{We use MATLAB matrix indexing notation here. For a given matrix $\mathbf{A} $, $\mathbf{A}(:, 1:k)$ and $\mathbf{A}(1:k,:)$ extract the first $k$ columns and rows of $\mathbf{A}$, respectively.}.  This gives us the explicit matrix norms of our rank $k$ approximation errors as:
\begin{equation}
    \|\mathbf{A}-\mathbf{A}_{k}\|_{2} = \sigma_{k+1}\quad  \text{and} \quad  \|\mathbf{A}-\mathbf{A}_{k}\|_{F} = \sqrt{\sum_{i = k +1}^{n}\sigma_{i}^{2}} .
    \label{eq:ey}
\end{equation}
However, classical algorithms to calculate the SVD are extremely costly.
So it is impractical and often impossible to perform the SVD on a sufficiently large matrix due to the prohibitive time complexity and current machine constraints.  There exists an alternative method, which is obtained from column pivoting QR (CPQR). However, CPQR can not guarantee the accuracy of the low-rank approximation.  If we allow the diagonal matrix $\mathbf{\Sigma}$ in equation (\ref{eq:svd}) to be an upper triangular matrix, $\mathbf{R}$, we obtain the URV algorithm.  If instead we choose a lower triangular matrix, $\mathbf{L}$, for the decomposition, we get the ULV algorithm sometimes called the QLP algorithm, \cite{stewart1999qlp, stewart1993updating}.  This is a middle ground, as URV is faster than SVD yet more accurate than CPQR.
Besides URV, there exist other deterministic methods such as rank-revealing QR (RRQR) \cite{gu1996efficient} and rank-revealing LU (RRLU) \cite{miranian2003strong, pan2000existence} for low-rank approximation. Both RRQR and RRLU use permutations of the matrix columns and rows, which will finally produces a leading submatrix that captures most of the product of the singular values of the original matrix. 
RRQR and RRLU are expensive; however, they do not directly compute a low-rank approximation of the matrix \cite{anderson2017efficient}. 

Randomized low-rank approximation algorithms are a relatively recent development \cite{halko2011finding}.  Compared with the deterministic methods, such as SVD, RRQR and RRLU, randomized methods are usually faster and maintain high accuracy. Generally, there are two classes of randomized low-rank approximation algorithms: sampling-based algorithms and random projection-based algorithms.  Sampling algorithms use randomly selected columns or rows based on sampling probabilities derived from the original matrix.  One then performs a deterministic algorithm, such as SVD, on the smaller subsampled problem \cite{drineas2006fast, drineas2006fast2, drineas2006fast3, drineas2017lectures,kannan2017randomized, mahoney2009cur, mahoney2011randomized}. In contrast to the sampling-based algorithms, the core idea behind the random projection-based methods is to project the high-dimensional space spanned by the columns of the matrix into a low-dimensional space where deterministic methods can be inexpensively applied. In this paper, we mainly discuss the random projection-based methods. 
Formally, for a given matrix $\mathbf{A} \in \mathbb{R}^{m \times n}$, we multiply it with a random Gaussian matrix $\mathbf{\Omega} \in \mathbb{R}^{n \times l}$, where $l = k + q$, $k$ is the target rank and $q$ is the oversampling parameter. Then $\mathbf{Y} = \mathbf{A}\mathbf{\Omega}$ will capture the most action (i.e.~capture $l$ largest singular values)  of $\mathbf{A}$ and more importantly the dimension of the space spanned by the columns of $\mathbf{Y}$ is much smaller. Suppose $\mathbf{Q} \in \mathbb{R}^{m\times l}$ is the orthogonal basis for the approximate space $\mathbf{Y}$.  Then we have the following random QB low-rank approximation
\begin{equation}
    \mathbf{A} \approx \mathbf{Q}\mathbf{Q}^{T}\mathbf{A} =  \mathbf{QB}.
    \label{eq:randqb}
\end{equation}
Here $\mathbf{B} = \mathbf{Q}^{T} \mathbf{A} \in \mathbb{R}^{k\times n}$, which is a smaller matrix on which we could perform some standard factorization such as the SVD. 

\begin{algorithm}[!htb] 
\begin{algorithmic}[1]
\REQUIRE Matrix $\mathbf{A} \in \mathbb{R}^{m \times n}$, a target rank $k$,  oversampling parameter $q$ and $l = k + q$,
$p \geq 0$.  
\ENSURE Orthogonal matrix $\mathbf{U}\in \mathbb{R}^{m \times n}$, $\mathbf{V} \in \mathbb{R}^{n \times n}$ and diagonal $\mathbf{\Sigma} \in \mathbb{R}^{n \times n}$ in an approximate rank-$(l)$ SVD of $\mathbf{A}$, such that $\mathbf{A} \approx \mathbf{U\Sigma V}^{T}$

 \STATE Draw a random $n \times l $ test matrix $\mathbf{\Omega}$; \label{rsvd:1}
 \STATE Form the matrix product  $\mathbf{Y} = (\mathbf{A}\mathbf{A}^{T})^{p}\mathbf{A}\mathbf{ \Omega} $ ; 
 \STATE Perform QR decomposition to obtain orthonormal basis $\mathbf{Q} = qr(\mathbf{Y}) $; \label{rsvd:3} \COMMENT{$\mathcal{C}_{qr}ml^{2}$}
 
\STATE Form the $l \times n$ matrix $\mathbf{B} = \mathbf{Q}^{T}\mathbf{A}$; \label{rsvd:pb} \COMMENT{$\mathcal{C}_{mm}mnl$}
 \STATE Form the SVD of the small matrix $\mathbf{B}: \mathbf{B} = \hat{\mathbf{U}}\mathbf{\Sigma V}^{T}$; \COMMENT{$\mathcal{C}_{svd}nl^{2}$}
 \STATE Form $\mathbf{U} = \mathbf{Q}\hat{\mathbf{U}}$;\COMMENT{$\mathcal{C}_{mm}ml^{2}$}
\end{algorithmic} 
\caption{A Basic Randomized SVD Algorithm}
\label{alg:rsvd}
\end{algorithm}
For a matrix whose singular values decay slowly, The above procedure may not provide good results. However, the power iteration $(\mathbf{A}\mathbf{A}^{T})^{p}\mathbf{A}$ can be applied to solve this problem, where $p$ is the exponent of the power iteration. Suppose $\mathbf{A}$ has the SVD decomposition (\ref{eq:svd}), then
\begin{equation}
  (\mathbf{AA}^{T})^{p}\mathbf{A} = \mathbf{U}\mathbf{\Sigma}^{2p+1}\mathbf{V}^{T}. 
  \label{eq:pi}
\end{equation}
Equation (\ref{eq:pi}) implies that the singular values of the matrix $(\mathbf{A}\mathbf{A}^{T})^{p}\mathbf{A}$ decay faster than the matrix $\mathbf{A}$. 
However, both of them have the same left and right eigenvectors. 
A basic randomized SVD algorithm  (RandSVD) is shown here as Algorithm \ref{alg:rsvd} \cite{halko2011finding}.
However, to avoid the rounding error of float point arithmetic obtained from performing the power iteration, reorthogonalization is needed. 
A typical way to perform the power iteration reorthogonalization is shown in Algorithm \ref{alg:basicpoweriterationsvd} \cite{halko2011finding,voronin2015rsvdpack}, where a reduced QR decomposition is applied in each reorthogonalization. An optimized version of reorthogonalization was proposed in \cite{li2017algorithm}, where the authors used LU factorization to replace QR every time except the last iteration. 
To solve the fixed precision low-rank approximation problem, the adaptive version of the randomized QB algorithm has been proposed to find the rank incrementally in \cite{halko2011finding, martinsson2016randomized} and was optimized in \cite{yu2018efficient}.
For the RandSVD Algorithm  \ref{alg:rsvd},  we need to access the matrix an even number times.  In some cases, the input matrix is enormous, and accessing the matrix is very expensive.  Randomized SVD algorithms that minimize accesses of the matrix have been proposed in \cite{bjarkason2019pass}. For the extremely large matrices, single-pass (single access) algorithms have also been studied recently. 
A general scheme for the single-pass algorithm was proposed in \cite{tropp2017practical}. With the single-pass algorithm, the matrix can processed via its random Gaussian projection. In \cite{halko2011finding}, the authors propose a single-pass algorithm, which uses two random Gaussian matrices to compress the original matrix and finally solve a linear equation. However, the single-pass algorithms in \cite{halko2011finding, tropp2017practical} are less accurate, to some extent.  For a matrix in row or column major format, a more accurate algorithm has been proposed in \cite{yu2018efficient}.

\begin{algorithm}[htb] 
	\begin{algorithmic}[1]
		\REQUIRE Given a $\mathbf{A} \in \mathbb{R}^{m \times n}$, $l < min(m,n)$ and
		$p \geq 0$.
		\ENSURE Orthonormal basis $\mathbf{Q} \in \mathbb{R}^{m \times l}$ for input matrix $\mathbf{A}$
		\STATE Generate random Gaussian matrix $\mathbf{\Omega}$ with size $n \times l$.
		\STATE [$\mathbf{Q}, \sim] = qr(\mathbf{A\Omega})$ \COMMENT{$\mathcal{C}_{qr}mnl  + \mathcal{C}_{qr}ml^{2}$}
		\FOR{$ i = 1:1:p$}
		\STATE [$\mathbf{Q}, \sim] = qr(\mathbf{A^{T}Q})$ \COMMENT{$\mathcal{C}_{qr}mnl  + \mathcal{C}_{qr}nl^{2}$}
		\STATE [$\mathbf{Q}, \sim] = qr(\mathbf{AQ})$ \COMMENT{$\mathcal{C}_{qr}mnl  + \mathcal{C}_{qr}ml^{2}$}
		\ENDFOR
	\end{algorithmic}
	\caption{Power Iteration Reorthogonalization for Algorithm \ref{alg:rsvd}}
	\label{alg:basicpoweriterationsvd}
\end{algorithm}

A randomized LU factorization algorithm (RandLU) for low-rank matrix approximation based on the RandSVD  was first proposed in \cite{shabat2018randomized}.
The authors give two reasons for their motivation for the  RandLU  decomposition as: (a) compared with SVD, LU decomposition is usually faster and also very efficient for sparse matrices with computation time related to the nonzero elements; (b) LU decomposition can be fully parallelized, which makes it very efficient in computation on modern hardware such as the GPU \cite{shabat2018randomized}. 
The overall algorithmic flow of RandLU is shown in the Algorithm \ref{alg:rlu}. In \cite{shabat2018randomized}, the authors showed the efficiency and accuracy of the algorithm \footnote{Code available at \cite{RandLUCode}}.   However,  costly matrix computations such as the pseudoinverse are needed in RandLU. In \cite{shabat2018randomized}, the authors demonstrated the speed superiority of the RandLU Algorithm \ref{alg:rlu} over RandSVD with relatively small matrix size.  Our tests show that RandLU can lose its advantage over RandSVD with larger matrices.  Besides, RandLU can only solve the fixed-rank low-rank approximation problem.  For the fixed precision low-rank approximation problem, RandLU will not work.  Also, RandLU is not a single-pass algorithm.

In this paper, we propose a pass-efficient randomized LU algorithm, called PowerLU, to solve the fixed-rank low-rank approximation problem. Compared with RandLU, PowerLU is  usually faster and also maintains high accuracy.  Moreover, PowerLU can access the matrix, $\mathbf{A}$, an arbitrary number of times. 
Next we propose a version of PowerLU, called PowerLU\_FP, for the fixed precision low-rank approximation problem. PowerLU\_FP is based on an efficient blocked adaptive rank determination Algorithm \ref{alg:adaptivewithoutupdating} proposed in this paper. 
Lastly, we propose a single-pass algorithm for low-rank approximation based on LU decomposition under the assumption that the matrix is stored in column-major or row-major order.  Experiments show that our proposed single-pass algorithm can achieve better results compared with the single-pass algorithm in \cite{halko2011finding}.

The rest of this paper is organized as follows. 
In \S \ref{sec:sec2}, we define some basic notation and review the existing RandLU algorithm.  We then further discuss the RandSVD and RandLU algorithms' commonality. 
Then we describe our proposed PowerLU algorithm and it's analysis in \S \ref{sec:sec3}.  In \S \ref{sec:sec4}, we explore a version of PowerLU algorithm, called PowerLU\_FP,  for the fixed precision low-rank approximation problem.  In \S \ref{sec:sec5}, the single-pass algorithm is discussed. In \S \ref{sec:sec6}, we present the numerical results from numerous experiments to demonstrate the efficiency of our proposed algorithms.

\section{Technical Preliminaries and Related Work}\label{sec:sec2}  

In this section, we first introduce some linear algebra basics needed in our paper. Then we briefly review the RandLU decomposition in \cite{shabat2018randomized}. 
To avoid the round-off error caused by computing the power iteration in RandLU, we employ a practical reorthogonalization procedure. 
Finally,  we give a discussion of the RandSVD and RandLU methods. 
Throughout the paper, we use the following notation: for any matrix $\mathbf{A}$, $\|\mathbf{A}\|$ denotes the spectral norm by default, which is the largest singular value of $\mathbf{A}$.  We use $\|\mathbf{A}\|_{F}$ for the Frobenius norm, which is $(\sum_{i,j}|a_{ij}|^{2})^{1/2}$. 
$\sigma_{k}$ signifies the $k$th largest singular value of $\mathbf{A}$.  To describe algorithms in this paper, we use MATLAB notation in our pseudocode, where $``lu(\cdot)"$ and $``qr(\cdot)"$ to denote the MATLAB builtin LU and QR factorizations. 

\subsection{Linear Algebra Basics}
In this section, we will briefly review the definitions and properties of the orthogonal projection and the pseudoinverse \cite{golub2012matrix}.

\subsubsection*{Orthogonal Projection}
An orthogonal projection is a linear transformation from a vector space to itself. For a matrix $\mathbf{A} \in \R^{m \times n}$, we use $Range(\mathbf{A})$ to denote the space spanned by the columns of $\mathbf{A}$. Suppose $\mathbf{A}$ has full column rank, then we denote the orthogonal projection of $Range(\mathbf{A})$ as $\mathbf{P}_{\mathbf{A}}$, which is defined as follows:
\begin{equation}
    \mathbf{P}_{\mathbf{A}} = \mathbf{A}(\mathbf{A}^{T}\mathbf{A})^{-1}\mathbf{A}^{T}.
    \label{eq:orthoprojector}
\end{equation}
It is easy to prove the following proprieties for $\mathbf{P}_{\mathbf{A}}$:
\begin{itemize}
    \item $\mathbf{P}_{\mathbf{A}}^{2} = \mathbf{P}_{\mathbf{A}}$.
    \item $Range(\mathbf{A}) = Range(\mathbf{P}_{\mathbf{A}})$.
    \item $\mathbf{P}_{\mathbf{A}} = \mathbf{A}\mathbf{A}^{T}$, if $\mathbf{A}$ has orthonormal columns.
\end{itemize}

\subsubsection*{Pseudoinverse}In linear algebra, the pseudoinverse of a matrix $\mathbf{A} \in \R^{m \times n}$ is the generalization of the inverse for non-square matrices. It is customary to denote the pseudoinverse of $\mathbf{A} \in \R^{m \times n}$ with a dagger, $\mathbf{A}^{\dagger}$. If $m \geq n$ and $\mathbf{A} \in \R^{m \times n}$ has full column rank, 
then the pseudoinverse of $\mathbf{A}^{m \times n}$ can be computed as
\begin{equation}
    \mathbf{A}^{\dagger} = (\mathbf{A}^{T}\mathbf{A})^{-1}\mathbf{A}^{T},
    \label{eq:pseudoinverse}
\end{equation}
where $\mathbf{A}^{\dagger} \in \R^{n \times m}$. 
$\mathbf{A}^{\dagger} \in \R^{n \times m}$ has the following properties:  
\begin{itemize}
    \item
    $(\mathbf{A}^{\dagger})^{T} = (\mathbf{A}^{T})^{\dagger}$ .
    \item If $\mathbf{A} \in \R^{m \times n}$ has full column rank, then
    $\mathbf{P}_{A} = \mathbf{A}\mathbf{A}^{\dagger}$ is an orthogonal  projection. 

\end{itemize}

\subsection{Randomized LU Decomposition}
\label{sec:rlu}
LU factorization is a fundamental problem in numerical linear algebra, and plays an important role in solving systems of linear equations. Compared with the QR and SVD decompositions, LU is generally faster.  In \cite{shabat2018randomized}, the authors developed RandLU decomposition for the fixed rank low-rank approximation problem.  RandLU decomposition can be applied to problems such as rank deficient least square, image reconstruction, and dictionary construction \cite{shabat2018randomized, rotbart2015randomized}. 

The RandLU algorithm is shown in Algorithm \ref{alg:rlu}, which produces matrices $\mathbf{P, Q,L, U}$ satisfying Theorem \ref{error: rlu}, where $\mathbf{P, Q}$ are permutation matrices, $\mathbf{L}$ is a lower triangular matrix and $\mathbf{U}$ is an upper triangular matrix when $p = 0$, where $p$ is the exponent of the power iteration in step 2. 
When $p > 0$, power iteration makes the RandLU algorithm  more accurate for a matrix with slowly decaying singular values. However, if we compute step 2 in floating-point arithmetic, rounding errors will overwhelm the small singular values compared to the spectral norm $\|\mathbf{A}\|$ \cite{halko2011finding, yu2018efficient}. 
To solve this problem, reorthogonalization with reduced QR when performing the power iteration was used \cite{halko2011finding}. In \cite{li2017algorithm}, an accelerated power iteration was proposed by replacing the QR with partial pivoting LU except for the last iteration. However, for randomized LU, there is no need for an orthonormal basis at the end, so one need not do the QR factorization at the end.
In this paper, we propose a power iteration reorthogonalization variant, which is shown as Algorithm \ref{alg:basicpoweriteration},  for the RandLU Algorithm \ref{alg:rlu}. In practice, we replace steps \ref{randlu:1}-\ref{randlu:3} in RandLU Algorithm \ref{alg:rlu} with Algorithm \ref{alg:basicpoweriteration}.  In \S \ref{sec:sec6},  we show that with reorthogonalization RandLU can achieve better accuracy in practice. 

\begin{remark}
In Algorithm \ref{alg:rlu}, we use partial pivoting LU instead of applying RRLU \cite{pan2000existence} to $\mathbf{Y}$. The authors pointed out that partial pivoting LU works well for most cases, and they also use partial pivoting LU in their own implementation \cite{RandLUCode}. 
In Algorithm \ref{alg:rlu}, we need to multiply $\mathbf{PA}$ by the pseudoinverse of $\mathbf{L}_{y}$. We use  (\ref{eq:pseudoinverse}) to compute $\mathbf{L}_{y}^{\dagger}$. 
\end{remark}

\begin{remark}
Power iteration has been applied to RandSVD, RandLU and our proposed PowerLU Algorithms. Power iteration reorthogonalization can remove rounding error caused by floating point arithmetic. However, in theory, reorthogonalization algorithms cannot improve the accuracy of randomized algorithms if exact arithmetic is used. The accuracy of the randomized algorithms is related the number of passes of the matrix, which in turn speeds the decay of the singular values.  

\end{remark}

\begin{theorem}[\textbf{Error Bound for RandLU}]
       \label{error: rlu}
    For a given matrix $\mathbf{A} \in \mathbb{R} ^{m \times n}$,  then, its randomized LU produced by Algorithm \ref{alg:rlu} with integers $k$ , $p = 0$ and $\ell (\ell \geq k)$ satisfies:
    \begin{equation}
    \|\mathbf{LU} - \mathbf{PAQ}\| \leq \left(2\sqrt{2n\ell\beta^{2}\gamma^{2} + 1} + 2 \sqrt{2n\ell}\beta \gamma (k(n-k) + 1)\right)\sigma_{k+1}(\mathbf{A}),
\end{equation}
with probability greater than or equal to
\begin{equation}
    \Theta =
   1-
    \frac{1}{\sqrt{2\pi (\ell-k+1)}}\left(\frac{e}{(\ell-k+1)\beta}\right)^{\ell-k+1} - \frac{1}{4(\gamma^{2}-1)\sqrt{\pi n \gamma^{2}}}\left(\frac{2\gamma^{2}}{e^{\gamma^{2}-1}}\right)^{n},
\end{equation}
where $\beta > 0$ and $\gamma > 1$. 
    \label{thm:error}

\end{theorem}
This is a result from \cite{shabat2018randomized}, and by probability we mean the measure on the random Gaussian matrix used in their algorithm.  We believe that we can prove similar bounds for our proposed algorithms in this paper based on our Theorem \ref{acc:t}.
\begin{algorithm}[htb] 
\begin{algorithmic}[1]
\REQUIRE Matrix $\mathbf{A} \in \mathbb{R}^{m \times n} $, desired rank $k$, $l \geq k$ number of columns to use, and $p \geq 0$.

\ENSURE Matrix: $\mathbf{P} \in \mathbb{R}^{m \times n}, \mathbf{Q} \in \mathbb{R}^{n \times n}, \mathbf{L} \in \mathbb{R}^{m \times n}, \mathbf{U} \in \mathbb{R}^{n \times n}$ such that $\mathbf{PAQ} \approx \mathbf{LU}$, where $\mathbf{P}, \mathbf{Q}$ are orthogonal permutation matrices, $\mathbf{L}$ and $\mathbf{U}$ are lower and upper triangular matrices, respectively. 
\STATE Generate a randomized Gaussian matrix $\mathbf{\Omega}$ with size $n\times l$;\label{randlu:1}
\STATE $\mathbf{Y}=  \mathbf{A}(\mathbf{A}^{T}\mathbf{A})^{p} \mathbf{\Omega}$;   \label{randlu:2}                
\STATE $[\mathbf{L}_{y}, \mathbf{U}_{y}, \mathbf{P}] = lu(\mathbf{Y})$; \label{randlu:3} \COMMENT{$ \mathcal{C}_{lu}ml^{2}$}
\STATE Truncate $\mathbf{L}_{y}$ and $\mathbf{U}_{y}$ by choosing the first $k$ columns and first $k$ rows, respectively, such that $\mathbf{L}_{y} =  \mathbf{L}_{y}(:,1:k)$ and $\mathbf{U}_{y} = \mathbf{U}_{y}(1:k,:)$; \label{randlu:step4}
\STATE $\mathbf{B} = \mathbf{L}_{y}^{\dagger}\mathbf{PA}$;\label{rlu:pb} \COMMENT{$\mathcal{C}_{mm}(mnk + 2mk^{2}) + \mathcal{C}_{inv}k^{3} $}
\STATE Apply LU decomposition to $\mathbf{B}$ with column pivoting $\mathbf{BQ} = \mathbf{L}_{b}\mathbf{U}_{b}$; 
\STATE $\mathbf{L} = \mathbf{L}_{y}\mathbf{L}_{b}$; \COMMENT{$\mathcal{C}_{mm}mk^{2}$}
\STATE $\mathbf{U} = \mathbf{U}_{b}$;
 \end{algorithmic}
\caption{Randomized LU Decomposition}
\label{alg:rlu}
\end{algorithm}

\begin{algorithm}[htb] 
\begin{algorithmic}[1]
\REQUIRE Given a $\mathbf{A} \in \mathbb{R}^{m \times n}$, $l < min(m,n)$ and
$p \geq 0$.
\ENSURE $\mathbf{P}  \in \mathbb{R}^{m \times m}$, $\mathbf{L}  \in \mathbb{R}^{m \times l}$ and $\mathbf{U}  \in \mathbb{R}^{l \times l}$ such that $\mathbf{P}^{T}\mathbf{LU} = \mathbf{A}(\mathbf{A}^{T}\mathbf{A})^{p}\mathbf{\Omega}$, where $\mathbf{\Omega}$ is the random Gaussian matrix generated in this algorithm. 
\STATE Generate random Gaussian matrix $\mathbf{\Omega}$ with size $n \times l$;
\STATE $[\mathbf{L}, \mathbf{U}, \mathbf{P}] = lu(\mathbf{A\Omega})$; \COMMENT{$\mathcal{C}_{mm}mnl  + \mathcal{C}_{lu}ml^{2}$}
\IF{$p>0$}
\STATE $\mathbf{L} = \mathbf{P}^{T}\mathbf{L}$;
\FOR{$i = 1:1:p$}
    \STATE $[\mathbf{L}, \sim] = lu(\mathbf{A}^{T}\mathbf{L})$;  \COMMENT{$\mathcal{C}_{mm}mnl + \mathcal{C}_{lu}nl^{2}$}
    \IF{$i == p$}
    \STATE $[\mathbf{L}, \mathbf{U}, \mathbf{P}] = lu(\mathbf{A}\mathbf{L}))$;   \COMMENT{$\mathcal{C}_{mm}mnl + \mathcal{C}_{lu}ml^{2}$}
    \ELSE
    \STATE $[\mathbf{L}, \sim] = lu(\mathbf{A}\mathbf{L})$;\COMMENT{$\mathcal{C}_{mm}mnl + \mathcal{C}_{lu}ml^{2}$}
    \ENDIF
\ENDFOR
\ENDIF
 \end{algorithmic}
\caption{Power Iteration Reorthogonalization for Algorithm \ref{alg:rlu}}
\label{alg:basicpoweriteration}
\end{algorithm}

To describe the time complexity for the algorithms listed in this paper, we use the same notation as in \cite{martinsson2016randomized}:  let $\mathcal{C}_{mm}, \mathcal{C}_{lu}, \mathcal{C}_{qr}$, $\mathcal{C}_{svd}$ and $\mathcal{C}_{inv}$ denote the scaling constants for the cost of executing a matrix-matrix multiplication, partial pivoting LU factorization, a full QR factorization, an SVD factorization, and a matrix inversion respectively. Specifically, we have:
\begin{itemize}
    \item Multiplying two matrices of size $m \times n$ and $n \times r$ costs $\mathcal{C}_{mm}mnr$
    \item Performing a partial pivoting LU factorization of a  matrix of size $m \times n$, with $m \geq n$, costs $\mathcal{C}_{lu}mn^{2}$
    \item  Performing a full QR factorization of a matrix of size $m \times n$, with $m \geq n$, costs $\mathcal{C}_{qr}mn^{2}$
    \item Performing an SVD factorization of a matrix of size $m \times n$, with $m \geq n$, costs $\mathcal{C}_{svd}mn^{2}$
    \item Inverting a matrix of size $n \times n$, costs $\mathcal{C}_{inv}n^{3}$
\end{itemize}

Here, we compare the time complexity for RandLU and RandSVD. 
Reorthogonalization of the power iteration will be employed to reduce the rounding error. So for both RandLU and RandSVD, we analyze their time complexity together with their corresponding power iteration reorthogonalization variants.  It is easy to see the time complexity for the RandLU Algorithm \ref{alg:rlu}, where the power iteration is computed by Algorithm \ref{alg:basicpoweriteration},  is given by the formula  (\ref{time:rlu}) with $p \geq 0$ as following: 
\begin{equation}
\begin{split}
\label{time:rlu}
    \mathcal{C}_{RandLU} &\sim  \mathcal{C}_{mm}mnl + \mathcal{C}_{lu}ml^{2}  + p\cdot\left(2 \mathcal{C}_{mm}mnl +  \mathcal{C}_{lu}ml^{2} + \mathcal{C}_{lu}nl^{2}\right) +  \\ &
    \mathcal{C}_{mm}(mnk + 2mk^{2}) + \mathcal{C}_{inv}k^{3} + \mathcal{C}_{mm}mk^{2}.
\end{split}
\end{equation}

According to Algorithm \ref{alg:rsvd}, the time complexity for RandSVD Algorithm \ref{alg:rsvd}, where the power iteration  is computed by Algorithm \ref{alg:basicpoweriterationsvd} \cite{halko2011finding} is as follows:
\begin{equation}
\begin{split}
    \mathcal{C}_{RandSVD} &\sim \mathcal{C}_{mm}mnl + \mathcal{C}_{qr}ml^{2} + p\cdot\left(2 \mathcal{C}_{mm}mnl +  \mathcal{C}_{qr}ml^{2} + \mathcal{C}_{qr}nl^{2}\right) + \\ & \mathcal{C}_{mm}mnl + \mathcal{C}_{svd}nl^{2} + \mathcal{C}_{mm}ml^{2}.
    \label{time:rsvdp}
\end{split}
\end{equation}

It is hard to tell which of the two is larger, (\ref{time:rlu}) or  (\ref{time:rsvdp}). When $p = 0$,the authors \cite{shabat2018randomized} showed that RandLU is faster than RandSVD with fixed matrix size $m = n = 3000$, and various target rank values. Our own tests with the same matrix size confirmed this.  However, our tests  show that RandLU will lose this advantage over RandSVD when we use larger matrices. Details can be found in \S \ref{sec:sec6}.

So far, we have introduced RandLU factorization shown in \cite{shabat2018randomized}, and we employed the  power iteration reorthogonalization for RandLU as shown in Algorithm \ref{alg:basicpoweriteration}. RandLU can solve the fixed rank low-rank approximation problem for a given rank.  
However, in some cases, it is impossible to obtain the rank of the matrix in advance, so now we will explore algorithms for the fixed precision low-rank approximation problem based on LU factorization.

\subsection{Discussion} In this section, we develop the intuition to help understand our proposed randomized LU algorithm.  We previously discussed the RandSVD and RandLU algorithms.  And we will use our understanding of RandSVD Algorithm \ref{alg:rsvd} to examine RandLU Algorithm \ref{alg:rlu}.

We can organize the RandSVD Algorithm \ref{alg:rsvd} into two stages as follows:
\begin{enumerate}
    \item \label{stages:1} Randomized Stage: Find an orthonormal basis $\mathbf{Q} \in \R^{m \times l}$ of the column space of $\mathbf{A} \in \R ^{m \times n}$ via random projection. Then we approximate $\mathbf{A} \approx \mathbf{P}_{\mathbf{Q}}\mathbf{A} = \mathbf{Q}\mathbf{Q}^{T}\mathbf{A}$.
    \item \label{stages:2}  Deterministic SVD Stage: Let $\mathbf{B} = \mathbf{Q}^{T}\mathbf{A}$, then perform a deterministic SVD on $\mathbf{B}$, which is much smaller than $\mathbf{A}$.
\end{enumerate}
RandLU is organized similarly.
First, RandLU produces a lower triangular matrix, $\mathbf{L}_{y}$.  This matrix is used with its psuedoinverse to create an orthogonal projection (a permuted) of matrix $\mathbf{A}$.  This is analogous to the RandSVD randomization stage:
\begin{equation}
   \mathbf{PA} \approx \mathbf{L}_{y}\mathbf{L}_{y}^{\dagger}\mathbf{PA}.  
   \label{eq:approxlulu}
\end{equation}
Then partial pivoting LU is performed on $\mathbf{B} = \mathbf{L}_{y}^{\dagger}\mathbf{PA}$, where $\mathbf{B}$ is much smaller than $\mathbf{A}$.  This is the analog of the deterministic stage in RandSVD. 


However, it is costly to compute equation (\ref{eq:approxlulu}) in RandLU, since the pseudoinverse of a lower triangular matrix is needed.  In Algorithm \ref{alg:rlu}, step \ref{randlu:2}, power iteration is performed as $\mathbf{A}(\mathbf{A}^{T}\mathbf{A})^{p}\mathbf{\Omega}$.  As we discussed, to eliminate rounding errors in power iteration, a reorthogonalization procedure is necessary.  As Algorithm \ref{alg:basicpoweriteration2} shows, we can obtain an orthonormal basis $\mathbf{V} \in \R^{n \times l}$ for $\mathbf{A}^{T}$ by performing reorthogonalization on $(\mathbf{A}^{T}\mathbf{A})^{p}\mathbf{\Omega}$ when $p \geq 1$. Then we get the approximation
\begin{equation}
    \mathbf{A} \approx \mathbf{A}\mathbf{V}\mathbf{V}^{T}.
    \label{eq:rightsingular}
\end{equation}

In this paper, we present a randomized LU algorithm, called PowerLU, which is based on the approximation (\ref{eq:rightsingular}). We named our proposed algorithm PowerLU because the power iteration exponent $p \geq 1$ is required by our algorithm. The basic idea of PowerLU is that we perform partial pivoting LU on $\mathbf{AV}$ to get lower and upper triangular matrices $\mathbf{L}_{1}$, $\mathbf{U}_{1}$. Let $\mathbf{B} = \mathbf{U}_{1}\mathbf{V}^{T}$, then we perform another partial pivoting LU on $\mathbf{B}^{T}$.  PowerLU does not need to compute a pseudoinverse and so is generally faster than RandLU. 
Another advantage of using (\ref{eq:rightsingular}) is that if $\mathbf{V}$ has sufficiently long columns, we can adaptively determine the rank, $k \leq l$, of $\mathbf{A}$ quite easily. 
We will discuss the details in the upcoming sections.

\section{PowerLU: A Method for the Fixed Rank Low-Rank Matrix Approximation}\label{sec:sec3}
In this section, we discuss a new randomized LU factorization, called PowerLU, for the fixed rank low-rank matrix approximation problem  (\ref{eq:lowrank}). 
Our proposed method is based on (\ref{eq:rightsingular}). 
It is generally faster than the RandLU and RandSVD without loss of accuracy. In some cases, accessing the matrix $\mathbf{A}$ can be expensive.  RandLU and RandSVD using power iteration access the matrix an even number of times. Our proposed PowerLU allows an arbitrary number of matrix accesses, which can reduce computation time when an odd number of accesses is sufficient. 

\subsection{PowerLU: A Low-Rank Matrix Approximation Method} In this section, we will present our proposed PowerLU method in detail. 

The PowerLU algorithm is shown in Algorithm \ref{alg:bpowerlu}.  
Given a  matrix $\mathbf{A} \in \mathbb{R}^{m\times n}$, suppose we want to find a low-rank approximation for $\mathbf{A}$ with target rank $k$.  First we generate a random Gaussian matrix $\mathbf{\Omega} \in \mathbb{R}^{n \times l}$ where $l \geq k$. In practice we choose $l = k + q$, where $q = 5$ or $10$ is the oversampling size.  This is needed to ensure that the final rank is at least $k$ with high probability.  

In next step, to deal with a matrix with singular values decaying slowly, we compute an orthonormal basis $\mathbf{V} \in \R^{n \times l}$ by using QR decomposition on 
\begin{equation}
    (\mathbf{A}^{T}\mathbf{A})^{p}\mathbf{\Omega} = \mathbf{VZ}.
    \label{eq:powerlu}
\end{equation}
So $\mathbf{A}$ can be approximated by
\begin{equation}
    \mathbf{A} \approx \mathbf{A}\mathbf{V}\mathbf{V}^{T}.
    \label{eq:approxlu}
\end{equation}
Then we form the partial pivoting LU decomposition of $\mathbf{AV}$
\begin{equation}
    \mathbf{L_{1}U_{1}} = \mathbf{PAV},
    \label{eq:plu}
\end{equation}
where $\mathbf{P}$ is a permutation matrix.  Since $\mathbf{L_{1}U_{1}}$ derive from $\mathbf{AV}$, we can obtain an approximation for $\mathbf{PA}$ as 
\begin{equation}
    \mathbf{P}\mathbf{A} \approx \mathbf{P}\mathbf{A}\mathbf{V}\mathbf{V}^{T}  = \mathbf{L_{1}}\mathbf{U_{1}}\mathbf{V}^{T}.
    \label{eq:firstlu}
\end{equation}
Let $\mathbf{B}  = \mathbf{U}_{1}\mathbf{V}^{T}$, and we use standard partial pivoting LU on $\mathbf{B}^{T}$ to obtain

\begin{equation}
\label{eq:secondlu}
    \mathbf{Q}\mathbf{B}^{T} =\mathbf{L}_{2}\mathbf{U}_{2}.
\end{equation}
Combining the above equations (\ref{eq:firstlu}) and (\ref{eq:secondlu}), we obtain
\begin{equation}
   \mathbf{P} \mathbf{A} \mathbf{Q} \approx  \mathbf{L}\mathbf{U},
   \label{eq:powerlufinal}
\end{equation}
where $\mathbf{L} = \mathbf{L}_{1}\mathbf{U}_{2}^{T}$ and $\mathbf{U} = \mathbf{L}_{2}^{T}$. 

\begin{algorithm}[htb] 
\begin{algorithmic}[1]
\REQUIRE Given $\mathbf{A} \in \mathbb{R}^{m\times n}$, desired rank  $k$, oversampling parameter $q$ and $l =  k + q$, and $p \geq 1$.
\ENSURE Matrix: $\mathbf{P}, \mathbf{Q}, \mathbf{L}, \mathbf{U}$ such that $\mathbf{PAQ} \approx \mathbf{LU}$, where $\mathbf{P}, \mathbf{Q}$ are orthogonal permutation matrices, $\mathbf{L} \in \mathbb{R}^{m \times k}$ and $\mathbf{U} \in \mathbb{R}^{k\times n}$ are the lower and upper triangular matrix, respectively.
\STATE Generate a randomized Gaussian matrix $\mathbf{\Omega}$ with size $n \times l$; \label{powerLU:step1}
\STATE Compute $(\mathbf{A}^{T}\mathbf{A})^{p}  \mathbf{\Omega}$  and its orthonormal basis $\mathbf{V}$; \label{powerLU:step2}
\STATE $\mathbf{Y} = \mathbf{A}\mathbf{V}(:,1:k)$;\COMMENT{$\mathcal{C}_{mm}mnk$}      
\STATE  $[\mathbf{L}_{1}, \mathbf{U}_{1}, \mathbf{P}] = lu(\mathbf{Y})$;\label{powerlu:step4}\COMMENT{$\mathcal{C}_{lu} mk^{2}$}
\STATE  $\mathbf{B} = \mathbf{U}_{1}\mathbf{V}(:,1:k)^{T} $; \COMMENT{$\mathcal{C}_{mm} nk^{2}$}
\STATE $[\mathbf{L}_{2}, \mathbf{U}_{2}, \mathbf{Q}] = lu(\mathbf{B}^{T})$; \COMMENT{$\mathcal{C}_{lu} nk^{2}$}
\STATE $\mathbf{L} = \mathbf{L}_{1}\mathbf{U}_{2}^{T}$; \COMMENT{$\mathcal{C}_{mm} mk^{2}$}
\STATE $\mathbf{U} = \mathbf{L}_{2}^{T}$; \label{powerlu:step8}
\end{algorithmic}
\caption{The PowerLU Algorithm}
\label{alg:bpowerlu}
\end{algorithm}

In (\ref{eq:secondlu}), we form $\mathbf{B}$ through matrix-matrix multiplication. Compared with RandLU, PowerLU is does not compute a pseudoinverse, and so is usually much faster.  In step \ref{powerLU:step2} of the PowerLU Algorithm \ref{alg:bpowerlu}, we need to perform the power iteration. As we discussed in \S \ref{sec:rlu}, we might lose some accuracy if we calculate it without reorthogonalization.
In practice we use Algorithm \ref{alg:basicpoweriteration2} to replace steps \ref{powerLU:step1}-\ref{powerLU:step2} in Algorithm \ref{alg:bpowerlu}. In Algorithm \ref{alg:bpowerlu}, we need to obtain the orthonormal basis $\mathbf{V}$. Therefore, Algorithm \ref{alg:basicpoweriteration2} for PowerLU is different with Algorithm \ref{alg:basicpoweriteration}.  

\subsection{Accuracy of Algorithm \ref{alg:bpowerlu}} 
The accuracy of the PowerLU Algorithm \ref{alg:bpowerlu} is the approximation error obtained from (\ref{eq:approxlu}). Theorem 4.1 in \cite{bjarkason2019pass} gives us an error bound for (\ref{eq:approxlu}), which we state here as Theorem \ref{thm:approxlu}. 

\begin{theorem}
\label{thm:approxlu}
Let $\mathbf{A}\in \R^{m \times n}$ with nonnegative singular values $\sigma_{1} \geq \sigma_{2} 
\geq \cdots \geq  \sigma_{min(m,n)} $, and let $k \geq 2$ be the target rank, $q \geq 2$ be an oversampling parameter, with $k + q 
\leq min(m,n)$. Draw a Gaussian random matrix $\mathbf{\Omega} \in \R^{n \times (k+q)}$ and set $\mathbf{Y} = (\mathbf{A}^{T}\mathbf{A})^{p}\mathbf{\Omega}$ for $p \geq 1$. Let $\mathbf{V} \in \R^{n \times (k + q)}$ be an orthonormal matrix which forms a basis for the range of $\mathbf{Y}$. Then
\begin{equation}
    \mathbb{E}[\|\mathbf{A} - \mathbf{A}\mathbf{V}\mathbf{V}^{T}\|]  \leq \left[\left(1 + \sqrt{\frac{k}{q-1}}\sigma_{k+1}^{2p}\right)\sigma_{k+1}^{2p} + \frac{e\sqrt{k+q}}{q}\left(\sum_{j>k}\sigma_{j}^{4p}\right)^{1/2}\right]^{1/(2p)}. 
\end{equation}
\end{theorem}

Using Theorem \ref{thm:approxlu}, PowerLU can achieve high accuracy results for a matrix with rapidly decaying singular values. Otherwise, we can improve the accuracy by employing a larger value of exponent, $p$, in the power iteration. 


\begin{remark}
In the PowerLU Algorithm \ref{alg:bpowerlu}, we set $p \geq 1$. Since if $p = 0$, Algorithm \ref{alg:basicpoweriteration2} cannot output a orthonormal basis of $\mathbf{A}$.

\end{remark}

\begin{algorithm}[htb] 
\begin{algorithmic}[1]
\REQUIRE Given a matrix $\mathbf{A} \in \mathbb{R}^{m \times n}$ , $l < min(m,n)$ and
 $p \geq 1$.
\ENSURE Orthonormal basis $\mathbf{V} \in \mathbb{R}^{m \times l}$ of the matrix $(\mathbf{A}\mathbf{A}^{T})^{p}\mathbf{\Omega}$ 
\STATE Generate a random Gaussian matrix $\mathbf{\Omega}$ with size $n \times l$;
\FOR{$i = 1:1:p$}
    \STATE $[\mathbf{\Omega}, \sim] = lu(\mathbf{A}\mathbf{\Omega})$;  \COMMENT{$\mathcal{C}_{mm}mnl + \mathcal{C}_{lu}ml^{2}$}
    \IF{$i == p$}
    \STATE $[\mathbf{V}, ~] = qr(\mathbf{A}^{T}\mathbf{\Omega}))$; \label{PowerIteration: Step5}  \COMMENT{$\mathcal{C}_{mm}mnl + \mathcal{C}_{qr}nl^{2}$}
    \ELSE
    \STATE $[\mathbf{\Omega}, \sim] = lu(\mathbf{A}^{T}\mathbf{\Omega})$;\COMMENT{$\mathcal{C}_{mm}mnl + \mathcal{C}_{lu}nl^{2}$}
    \ENDIF
\ENDFOR
 \end{algorithmic}
\caption{Power Iteration Reorthogonalization for Algorithm \ref{alg:bpowerlu}}
\label{alg:basicpoweriteration2}
\end{algorithm}

\subsection{Relationship of PowerLU to RandLU}
So far we have discussed both RandLU and PowerLU. In practice, reorthogonalization of the power iteration is used in both algorithms. The time complexity for the PowerLU Algorithm \ref{alg:bpowerlu}, where the power iteration is computed by Algorithm \ref{alg:basicpoweriteration2} is 

 \begin{equation}
 \label{eq:powerlutimecost}
 \begin{split}
     \mathcal{C}_{PowerLU} & \sim  (p-1) \cdot \left[\left( 2\mathcal{C}_{mm}mnl + \mathcal{C}_{lu}ml^{2} + \mathcal{C}_{lu}nl^{2}\right)\right]  + 2\mathcal{C}_{mm}mnl \\ &+ \mathcal{C}_{lu}ml^{2} + \mathcal{C}_{qr}nl^{2} + \mathcal{C}_{mm} ( mnk + nk^{2} + mk^{2})  +   \mathcal{C}_{lu} (nk^{2} + mk^{2}).
\end{split}
 \end{equation}

So the difference in time complexity between RandLU and PowerLU is
 
 \begin{equation}
 \label{eq:diff}
 \begin{split}
 \mathcal{C}_{RandLU} - \mathcal{C}_{PowerLU} 
    &= \mathcal{C}_{mm} mnl + \mathcal{C}_{lu} ml^{2} +    (\mathcal{C}_{lu} -  \mathcal{C}_{qr})nl^{2} \\ 
    &+ \mathcal{C}_{mm}(2mk^{2}- nk^{2}) + \mathcal{C}_{inv}k^{3}- \mathcal{C}_{lu}(nk^{2} + mk^{2}).
 \end{split}
 \end{equation}

 \begin{remark}
 For RandLU and PowerLU, we assume the input matrix $\mathbf{A}\in \mathbb{R}^{m \times n}$, where $m \geq n$.  Since $l \approx k$  equation (\ref{eq:diff}) will most likely be positive, which means that our proposed PowerLU algorithm is less costly than RandLU. Our numerical tests confirm this for dense matrices. 
 \end{remark}

Theorem \ref{acc:t} below states the relation between the PowerLU and RandLU error bounds in the manner shown in Theorem \ref{error: rlu}. 

\begin{theorem}
Let $\mathbf{A}$ be an $m \times n$ matrix, let $l$ be a positive integer such that $l < min(m, n)$, let $p$ be a positive integer, and let $\mathbf{\Omega}$ be a random Gaussian matrix of size $n \times l$.  Let $\mathbf{P}_{p}\mathbf{A}\mathbf{Q}_{p} \approx \mathbf{L}_{p}\mathbf{U}_{p}$ be the factorization resulting from the PowerLU Algorithm \ref{alg:bpowerlu}.   Let $\mathbf{P}_{r}\mathbf{A}\mathbf{Q}_{r} \approx \mathbf{L}_{r}\mathbf{U}_{r}$ be the factorization resulting from the RandLU Algorithm \ref{alg:rlu}. Suppose the rank of the matrix is at least $l$ and both algorithms are performed in exact arithmetic.  Then we have:
\begin{equation}
    Range(\mathbf{P}_{r}^{T}\mathbf{L}_{r}) = Range(\mathbf{P}_{p}^{T}\mathbf{L}_{p}).
\end{equation}
  \label{acc:t}
\end{theorem}
\begin{proof}
First, we notice from Algorithm \ref{alg:bpowerlu} that we have the following unpivoted QR decomposition
\begin{equation}
    (\mathbf{A}^{T}\mathbf{A})^{p} \mathbf{\Omega}= \mathbf{V}\mathbf{Z},
\end{equation}
where $\mathbf{V}$ is the orthogonal basis for $(\mathbf{A}^{T}\mathbf{A})^{p}\mathbf{\Omega}$. 
Then we have 
\begin{equation}
    \mathbf{Y} = \mathbf{A}(\mathbf{A}^{T}\mathbf{A})^{p} \mathbf{\Omega} = \mathbf{A}\mathbf{VZ} = \mathbf{P}_{p}^{T}\mathbf{L}_{1}\mathbf{U}_{1}\mathbf{Z}.
    \label{eq:yequals}
\end{equation}
From the RandLU Algorithm \ref{alg:rlu}, we have the following: 
\begin{equation}
    \mathbf{P}_{r} \mathbf{Y} = \mathbf{L}_{y}\mathbf{U}_{y}, 
\end{equation}
which means $\mathbf{U}_{y}$ is invertable, and so
\begin{equation}
     \mathbf{L}_{y} = \mathbf{P}_{r} \mathbf{Y} \mathbf{U}_{y}^{-1}.
\end{equation}
Then we replace $\mathbf{Y}$ with the equation (\ref{eq:yequals}):
\begin{equation}
     \mathbf{L}_{y} = \mathbf{P}_{r} \mathbf{P}_{p}^{T}\mathbf{L}_{1}\mathbf{U}_{1}\mathbf{Z}\mathbf{U}_{y}^{-1}.
\end{equation}
Since  the matrices $\mathbf{U}_{1}, \mathbf{Z}, \mathbf{U}_{y}$ are non-singular matrices with size $l\times l$, we now have
\begin{equation}
Range(\mathbf{P}_{r}^{T}\mathbf{L}_{y}) = Range(\mathbf{P}_{p}^{T}\mathbf{L}_{1}).
\end{equation}
The RandLU Algorithm \ref{alg:rlu} implies we have $\mathbf{L}_{r} = \mathbf{L}_{y}\mathbf{L}_{b}$, and similarly the PowerLU Algorithm \ref{alg:bpowerlu} implies $\mathbf{L}_{p} = \mathbf{L}_{1}\mathbf{U}_{2}^{T}$. Since  both $\mathbf{L}_{b}, \mathbf{U}_{2}$ are non-singular, we arrive at 
\begin{equation}
    Range(\mathbf{P}_{r}^{T}\mathbf{L}_{r}) = Range(\mathbf{P}_{p}^{T}\mathbf{L}_{p}).
\end{equation}
\end{proof}


Theorem \ref{acc:t} established the relationship between RandLU and PowerLU. We obtain the following equation (\ref{eq:eqlu}) by using Theorem \ref{acc:t}.  So all the theoretical analysis for randomized LU \cite{shabat2018randomized} can be applied to the PowerLU decomposition: 
\begin{equation}
 \mathbf{P}_{r}^{T}\mathbf{L}_{r}\mathbf{L}_{r}^{\dagger}\mathbf{P}_{r}\mathbf{A} =  \mathbf{P}_{p}^{T}\mathbf{L}_{p}\mathbf{L}_{p}^{\dagger}\mathbf{P}_{p}\mathbf{A} .
 \label{eq:eqlu}
\end{equation}

We observe that the RandLU Algorithm \ref{alg:rlu} accesses the matrix $\mathbf{A}$ a total of $2p + 2$ times.  In much of the literature, accessing $\mathbf{A}$ is often referred to as a pass of $\mathbf{A}$.  However, only $2p + 1$ passes of $\mathbf{A}$ are needed in PowerLU. Therefore, for the same value $p$, PowerLU will have one less pass than RandLU and will therefore be less accurate. In the rest of the section, we will extend the PowerLU Algorithm \ref{alg:bpowerlu} so that works with any number of passes of the matrix $\mathbf{A}$. 

\subsection{Generalized PowerLU Decomposition}
As we discussed in the last section, PowerLU has an odd number passes of the matrix $\mathbf{A}$. For some cases, accessing the matrix will be expensive. However, to increase the accuracy of the computation, we have to increment by 2 passes each time due to multiplication by $\mathbf{A}^{T}\mathbf{A}$. In \cite{bjarkason2019pass}, the authors proposed the generalized randomized SVD algorithm, which allows any number of passes $v \geq 2$ by using generalized randomized subspace iteration.  A similar algorithm was proposed in \cite{feng2018fast}. Here we use $v$ to denote the number of passes of $\mathbf{A}$.  In this section, we will introduce the generalized PowerLU algorithm which can have any number passes $v \geq 2$ of $\mathbf{A}$.

The idea for the generalized PowerLU is very simple. Let $p = \lfloor\frac{v-1}{2}\rfloor$ when $v$ is odd in Algorithm \ref{alg:bpowerlu}. We modify steps \ref{powerLU:step1}-\ref{powerLU:step2} of Algorithm \ref{alg:bpowerlu} when $v \geq 2$ is even as follows:
\begin{enumerate}
    \item Generate a random Gaussian matrix $\mathbf{\Omega} \in \R^{m \times l}$;
    \item Compute $(\mathbf{A}^{T}\mathbf{A})^{p}  \mathbf{A}^{T}\mathbf{\Omega}$  and its orthonormal basis $\mathbf{V}$; 
\end{enumerate}
 We describe the power iteration reorthogonalization procedure for the generalized PowerLU in Algorithm \ref{alg:gpowerlu}.   Here we need only replace steps \ref{powerLU:step1}-\ref{powerLU:step2} in Algorithm \ref{alg:bpowerlu} by Algorithm \ref{alg:gpowerlu}, which allows an arbitrary number of passes $v \geq 2$ of the matrix $\mathbf{A}$.  
 
 When $v$ is odd, then the error bound for generalized PowerLU can be estimated by Theorem \ref{thm:approxlu} if we set $p = \lfloor\frac{v-1}{2}\rfloor$. When $v$ is even, the error bound for the generalized PowerLU is given by Theorem \ref{thm:approxlu2} \cite{bjarkason2019pass}.

\begin{theorem}
\label{thm:approxlu2}
Let $\mathbf{A}\in \R^{m \times n}$ with nonnegative singular values $\sigma_{1} \geq \sigma_{2} 
\geq \cdots \geq \sigma_{min(m,n)}$.  Let $k \geq 2$ be the target rank, $q \geq 2$ be an oversampling parameter, with $k + q 
\leq min(m,n)$. Draw a random Gaussian matrix $\mathbf{\Omega} \in \R^{n \times (k+q)}$ and set $\mathbf{Y} = (\mathbf{A}^{T}\mathbf{A})^{p}\mathbf{A}^{T}\mathbf{\Omega}$ for $p \geq 1$. Let $\mathbf{V} \in \R^{n \times (k + q)}$ be an orthonormal matrix which forms the basis for the range of $\mathbf{Y}$. Then
\begin{equation}
\begin{split}
    & \mathbb{E}[\|\mathbf{A} - \mathbf{A}\mathbf{V}\mathbf{V}^{T}\|]  = \mathbb{E}[\|\mathbf{A}^{T} - \mathbf{V}\mathbf{V}^{T}\mathbf{A}^{T}\|] \\& 
    \leq
    \left[\left(1 + \sqrt{\frac{k}{q-1}}\sigma_{k+1}^{2p+1}\right)\sigma_{k+1}^{2p+1} + \frac{e\sqrt{k+q}}{q}\left(\sum_{j>k}\sigma_{j}^{2(2p+1)}\right)^{1/2}\right]^{1/(2p+1)}. 
\end{split}
\end{equation}
\end{theorem}
\begin{algorithm}[htb] 
\begin{algorithmic}[1]
\REQUIRE Given a matrix $\mathbf{A}\ in \R^{m \times n}$, $l < min(m,n)$ and
$v \geq 2$.
\ENSURE Orthonormal basis $\mathbf{V} \in \mathbb{R}^{n \times l}$ for matrix $(\mathbf{A}^{T}\mathbf{A})^{\lfloor\frac{v-1}{2}\rfloor}\mathbf{A}^{T}\mathbf{\Omega}$ or $(\mathbf{A}^{T}\mathbf{A})^{\lfloor\frac{v-1}{2}\rfloor}\mathbf{\Omega}$.
\IF{$v$ even}
    \STATE Generate a random Gaussian matrix $\mathbf{\Omega}$ with size $m \times l$;
    \IF{$v > 2$}
    \STATE $[\mathbf{V}, \sim] = lu(\mathbf{A^{T}\mathbf{\Omega}})$; \COMMENT{$\mathcal{C}_{mm}mnl + \mathcal{C}_{lu}nl^{2}$} 
    \ELSE
    \STATE $[\mathbf{V}, \sim] = qr(\mathbf{A^{T}\Omega})$;
    \COMMENT{$\mathcal{C}_{mm}mnl + \mathcal{C}_{qr}nl^{2}$} 
    \ENDIF
\ELSE
 \STATE Generate a random Gaussian matrix $\mathbf{V}$ with size $n \times l$;
\ENDIF
\FOR{$i = 1:\lfloor\frac{v-1}{2}\rfloor$}
    \STATE $[\mathbf{V}, \sim] = lu(\mathbf{AV})$; \COMMENT{$\mathcal{C}_{mm}mnl + \mathcal{C}_{lu}ml^{2}$}      
    \IF{$i ==\lfloor\frac{v-1}{2}\rfloor$}
    \STATE $[\mathbf{V}, \sim] = qr(\mathbf{A}^{T}\mathbf{V})$; \COMMENT{$\mathcal{C}_{mm}mnl + \mathcal{C}_{qr}nl^{2}$} 
    \ELSE
    \STATE $[\mathbf{V}, \sim] = lu(\mathbf{A}^{T}\mathbf{V})$;  \COMMENT{$\mathcal{C}_{mm}mnl + \mathcal{C}_{lu}nl^{2}$}            
    \ENDIF
\ENDFOR
 \end{algorithmic}
\caption{Generalized Power Iteration Reorthogonalization for Algorithm \ref{alg:bpowerlu}}
\label{alg:gpowerlu}
\end{algorithm}

Its easy to see that if the power iteration  is computed by Algorithm \ref{alg:gpowerlu},  the cost for the generalized PowerLU is the same as PowerLU when $v$ is odd, 
When $v$ is even, an extra factorization (LU or QR) and an extra matrix-matrix multiplication will be required.  
Subsequently PowerLU will refer to the generalized PowerLU decomposition just presented.
Numerical experiments to show the efficiency of our proposed algorithm are presented in \S \ref{sec:sec6}.  In the next section, we will discuss how to create a version of PowerLU to solve the fixed precision problem.

\section{PowerLU\_FP: A Method for Fixed Precision Low-Rank Matrix Approximation} \label{sec:sec4}
We now introduce a version of the PowerLU algorithm
used to solve the fixed precision low-rank matrix approximation problem. PowerLU\_FP is based on an efficient blocked adaptive rank determination algorithm proposed in this section. 
Our proposed blocked adaptive rank determination algorithm is inspired by the blocked random QB algorithm in \cite{martinsson2016randomized}.  In \cite{martinsson2016randomized}, the authors try to build an orthonormal basis $\mathbf{Q}$ of $\mathbf{A} \in \R^{m \times n}$ incrementally so that for a given tolerance, $\epsilon$,
\begin{equation}
    \|\mathbf{A} - \mathbf{Q}\mathbf{Q}^{T}\mathbf{A}\|_{F} \leq \epsilon.
    \label{eq:adaperror}
\end{equation}
In \cite{yu2018efficient}, the authors present an improved version of the blocked random QB algorithm, which uses an inexpensive error calculation as the stopping criterion. They compute a sufficiently large orthonormal basis, $\mathbf{V}$, for $\mathbf{A}^{T}$ when the power iteration exponent $p \geq 1$.
However, their algorithm also tries to build $\mathbf{Q}$ such that (\ref{eq:adaperror}) holds. They fail to use the properties of (\ref{eq:rightsingular}) to determine the rank of $\mathbf{A}$.

In rest of this section, we will partition the orthonormal basis of $\mathbf{A}^{T}$ into blocks, and propose a new blocked adaptive rank determination Algorithm \ref{alg:adaptivewithoutupdating}. We organize PowerLU\_FP to solve the fixed precision problem based on Algorithm \ref{alg:adaptivewithoutupdating}.  

\subsection{Blocked Adaptive Rank Determination for the Fixed Precision Problem}
\label{sec:blocking}

We will describe our blocked adaptive rank determination algorithm. 
The partition is inspired by the blocked random QB algorithm in \cite{martinsson2016randomized}.  

For a given matrix $\mathbf{A}\in \R^{m \times n}$ and target rank $k$, we compute an orthonormal basis, $\mathbf{V} \in \R ^{n \times l} $, for $\mathbf{A}^{T}$,  by Algorithm \ref{alg:gpowerlu} with sufficient large $l$, where we set $l =k = sb$.   Here we introduce $b$ as the block size, and $s$ as the number of blocks.  Then, we can partition the matrix $\mathbf{V}$ into blocks $\{\mathbf{V}_{j}\}_{j=1}^{s}$, each of size $n \times b$, 
\begin{equation}
     \mathbf{V} = \left[\mathbf{V}_{1}, \mathbf{V}_{2}, \cdots, \mathbf{V}_{s}\right].
     \label{eq:pv}
\end{equation}
To determine the rank adaptively, we should build a block version of the method to compute an approximate error 
\begin{equation}
    \|\mathbf{A} -\mathbf{A}\mathbf{V}\mathbf{V}^{T}\|_{F} .
    \label{eq:apperror}
\end{equation}
Initially, we set
\begin{equation}
\label{eq:b1}
    \mathbf{A}^{(0)} = \mathbf{A},
\end{equation}
and iterate for $i = 1,2,\cdots, s$ as follows
\begin{equation}
        \label{eq:b2}
        \mathbf{A}^{(i)}  =    \mathbf{A}^{(i-1)} -  \mathbf{A}^{(i-1)} \mathbf{V}_{i}\mathbf{V}_{i}^{T}.
    \end{equation}
Equations (\ref{eq:b1})-(\ref{eq:b2}) show us how to compute the error (\ref{eq:apperror}) via blocks of $\mathbf{V}$ . 

Next, we need to verify
\begin{equation}
    \|\mathbf{A}^{(s)}\|_{F} = \|\mathbf{A} -\mathbf{A}\mathbf{V}\mathbf{V}^{T}\|_{F}.
    \label{eq:equalequal}
\end{equation}

We state Theorem \ref{thm:22} to demonstrate (\ref{eq:equalequal}).
\begin{theorem}
\label{thm:22}
Let $\mathbf{A}\in \R^{ m\times n}$, $b$ is the block size, and $s$ is the number of blocks.  Suppose that the rank of $\mathbf{A}$ is at least $sb$. Let $\mathbf{V}$ be the orthonormal basis for $\mathbf{A}^{T}$, and we partition it as in (\ref{eq:pv}), with each $\mathbf{V}_{i}$ of size $n \times b$. Let $\{ \mathbf{A}^{(j)}\}_{j=1}^{i}$  be defined by 
(\ref{eq:b1})-(\ref{eq:b2}). We set:

\begin{equation}
    \label{eq:projection}
    \mathbf{P}_{i} = \sum_{j=1}^{i}\mathbf{V}_{j}\mathbf{V}_{j}^{T}.
\end{equation}
Then for every $i = 1,2,\cdots s$, we have
\begin{enumerate}[label=(\alph*)]
    \item The $\mathbf{P}_{i} $ are orthogonal projections\label{thm:aa}.
    \item $\mathbf{A}^{(i)} = \mathbf{A}(\mathbf{I}-\mathbf{P}_{i})$. \label{thm:bb} 
 
\end{enumerate}
\end{theorem}

\begin{proof}
The proof is as follows.

\begin{enumerate}
    \item To prove \ref{thm:aa}, we need to prove that
    \begin{equation}
        \mathbf{P}_{i}^{2} = \mathbf{P}_{i}.
    \end{equation}
    Since $\mathbf{V}$ has orthonormal columns,  $\mathbf{V}_{j}^{T}\mathbf{V}_{k} 
    = \delta_{jk} \mathbf{I}$. Therefore,
    \begin{equation}
        \mathbf{P}_{i}^{2} = (\sum_{j=1}^{i}\mathbf{V}_{j}\mathbf{V}_{j}^{T})^{2} = \mathbf{P}_{i}. 
    \end{equation}
    So \ref{thm:aa} is proved. 
    \item We prove \ref{thm:bb} by induction on $i$. Suppose \ref{thm:bb} holds for $i-1$, then
    \begin{equation}
        \mathbf{A}^{(i-1)} = \mathbf{A}(\mathbf{I}-\mathbf{P}_{i-1}).
    \end{equation}
    We can obtain $\mathbf{A}^{(i)}$ by (\ref{eq:b2})
    \begin{equation}
    \begin{split}
        \mathbf{A}^{(i)} &=\mathbf{A}^{(i-1)}-  \mathbf{A}^{(i-1)}\mathbf{V}_{i}\mathbf{V}_{i}^{T} \\&=
        \mathbf{A}^{(i-1)}(\mathbf{I}- \mathbf{V}_{i}\mathbf{V}_{i}^{T})
        \\&=  \mathbf{A}(\mathbf{I}-\mathbf{P}_{i-1})(\mathbf{I} - 
        \mathbf{V}_{i}\mathbf{V}_{i}^{T}) \\&=
        \mathbf{A}(\mathbf{I}-(\mathbf{V}_{i}\mathbf{V}_{i}^{T} + \mathbf{P}_{i-1})) = \mathbf{A}(\mathbf{I}-\mathbf{P}_{i}).
    \end{split}
    \end{equation}
    The last stop follows from (\ref{eq:projection}), and so \ref{thm:bb} is proved. 
\end{enumerate}

\end{proof}
Using Theorem \ref{thm:22}, we verify (\ref{eq:equalequal}) as follows:
\begin{equation}
    \|\mathbf{A}^{(s)}\|_{F} = \|\mathbf{A} - \mathbf{A}\mathbf{P}_{s}\|_{F} = \|\mathbf{A} - \mathbf{A}\sum_{j=1}^{s}\mathbf{V}_{j}\mathbf{V}_{j}^{T}\|_{F} = \|\mathbf{A} - \mathbf{A}\mathbf{V}\mathbf{V}^{T}\|_{F}.
\end{equation}
Now we can solve the fixed precision low-rank approximation for a given tolerance, $\epsilon$.  We now use the remainder of  $\mathbf{A}$, $\mathbf{A}^{(i)}$, to exit the loop when the tolerance is satisfied. We stop the computation when $\|\mathbf{A}^{(i)}\|_{F}^2 < \epsilon^{2}\|\textbf{A}\|_{F}^{2}$.  We can use steps (\ref{eq:b1}) - (\ref{eq:b2}) to determine the rank of $\mathbf{A}$. However, we need to update $\mathbf{A}^{(i)}$.  This is costly, especially when $\mathbf{A}$ is a sparse matrix (fill-in will occur when updating the matrix) \cite{yu2018efficient}. In \cite{yu2018efficient}, the authors proposed an error indicator to avoid updating $\mathbf{A}^{(i)}$.  In this section, we will introduce a similar error indicator, which can estimate the error quickly without updating.
\begin{theorem}
Let $\mathbf{A}$ be an $m \times n$ matrix, and $\mathbf{V}$ is an $n\times b$ orthogonal column matrix (i.e., $\mathbf{V}^{T}\mathbf{V} = \mathbf{I}_{b}$ ). Suppose $b < n$ and $\mathbf{B} = \mathbf{AV}$. Then 
\begin{equation}
    \|\mathbf{A} - \mathbf{B}\mathbf{V}^{T}\|_{F}^{2} = \|\mathbf{A}\|_{F}^{2} - \|\mathbf{B}\|_{F}^{2}. 
    \label{eq:errorindicator}
\end{equation}
\label{thm:errorindicator}
\end{theorem}
\begin{proof}
Our proof is similar to that of Theorem 1 in \cite{yu2018efficient}. By the properties of the Frobenius norm, for any matrix $\mathbf{M}$, we have
\begin{equation}
    \|\mathbf{M}\|_{F}^{2} = tr(\mathbf{M}\mathbf{M}^{T}), 
\end{equation}
where $tr(\cdot)$ is the trace of the matrix. Then we compute
\begin{equation}
\begin{split}
    (\mathbf{A} - \mathbf{B}\mathbf{V}^{T})(\mathbf{A} - \mathbf{B}\mathbf{V}^{T})^{T} &= (\mathbf{A} - \mathbf{B}\mathbf{V}^{T})(\mathbf{A}^{T} - \mathbf{V}\mathbf{B}^{T}) \\&= \mathbf{A}\mathbf{A}^{T} - \mathbf{A}\mathbf{V}\mathbf{B}^{T} - \mathbf{B}\mathbf{V}^{T}\mathbf{A}^{T} + \mathbf{B}\mathbf{V}^{T}\mathbf{V}\mathbf{B}^{T}\\&=
    \mathbf{A}\mathbf{A}^{T} - \mathbf{B}\mathbf{B}^{T} - \mathbf{B}\mathbf{B}^{T} + \mathbf{B}\mathbf{B}^{T} \\
    &= \mathbf{A}\mathbf{A}^{T} - \mathbf{B}\mathbf{B}^{T}.
    \end{split}
\end{equation}
Since the trace is a linear operator, taking the trace of the above proves (\ref{eq:errorindicator}). 
\end{proof}

\begin{theorem}
\label{thm:errorerror}
After the $i$th iteration of the loop (\ref{eq:b1}) - (\ref{eq:b2}), the error, $E= \|\mathbf{A}^{(i)}\|_{F}^{2}$, is given by
   \begin{equation}
    E = \|\mathbf{A}\|_{F}^{2}  - \sum_{j = 1}^{i}\|\mathbf{A}\mathbf{V}_{j}\|_{F}^{2}.
    \label{eq:errorerror}
\end{equation}
\end{theorem}

\begin{proof}
The proof is as follows.
\begin{equation}
\begin{split}
     E &= \|\mathbf{A}^{(i)}\|_{F}^{2} =  \|\mathbf{A}^{(i-1)} -  \mathbf{A}^{(i-1)}\mathbf{V}_{i}\mathbf{V}^{T}_{i}\|_{F}^{2} \\&=
    \|\mathbf{A}^{(i-1)}\|_{F}^{2} - \|\mathbf{A}^{(i-1)}\mathbf{V}_{i}\|_{F}^{2} \\&=
    \|\mathbf{A}^{(i-2)} -  \mathbf{A}^{(i-2)}\mathbf{V}_{i-1}\mathbf{V}^{T}_{i-1}\|_{F}^{2} - \|\mathbf{A}^{(i-1)}\mathbf{V}_{i}\|_{F}^{2} \\&=
    \|\mathbf{A}^{(i-2)}\|_{F}^{2} - \|\mathbf{A}^{(i-2)}\mathbf{V}_{i-1}\|_{F}^{2} - \|\mathbf{A}^{(i-1)}\mathbf{V}_{i}\|_{F}^{2} 
    \\& = \cdots 
    \\&= \|\mathbf{A}\|_{F}^{2}  - \sum_{j = 1}^{i}\|\mathbf{A}^{(j-1)}\mathbf{V}_{j}\|_{F}^{2}.
\end{split}
\label{eq:err1}
\end{equation}

In (\ref{eq:err1}), we need to prove that $\|\mathbf{A}^{(j-1)}\mathbf{V}_{j}\|_{F}^{2}$ = $\|\mathbf{A}\mathbf{V}_{j}\|_{F}^{2}$ when $j \leq i$. We have
\begin{equation}
    \begin{split}
        \|\mathbf{A}^{(j-1)}\mathbf{V}_{j}\|_{F}^{2}  &= \|(\mathbf{A}^{(j-2)} - \mathbf{A}^{(j-2)}\mathbf{V}_{j-1}\mathbf{V}_{j-1}^{T}) \mathbf{V}_{j}\|_{F}^{2} \\&=
         \|\mathbf{A}^{(j-2)}\mathbf{V}_{j} - \mathbf{A}^{(j-2)}\mathbf{V}_{j-1}\mathbf{V}_{j-1}^{T} \mathbf{V}_{j}\|_{F}^{2} 
    \end{split}
    \label{eq:ajvi}
\end{equation}
In equation (\ref{eq:ajvi}), $ \mathbf{A}^{(j-2)}\mathbf{V}_{j-1}\mathbf{V}_{j-1}^{T} \mathbf{V}_{j}  = \mathbf{0}$  since $\mathbf{V}$ has orthonormal columns, and $\mathbf{V}_{j-1}^{T} \mathbf{V}_{j} = \mathbf{0}$. Then

\begin{equation}
        \|\mathbf{A}^{(j-1)}\mathbf{V}_{j}\|_{F}^{2}  
        =
         \|\mathbf{A}^{(j-2)}\mathbf{V}_{j}\|_{F}^{2} = \cdots = \|\mathbf{A}\mathbf{V}_{j}\|_{F}^{2}.
\end{equation}

\end{proof}

Theorem \ref{thm:errorerror}  is an error indicator without the need to update the remainder of $\mathbf{A}$. The efficient blocked adaptive rank determination algorithm  without updating is shown in Algorithm \ref{alg:adaptivewithoutupdating}.
\begin{algorithm}[htb] 
\begin{algorithmic}[1]
\REQUIRE Given $\mathbf{A} \in \mathbb{R}^{m\times n}$, desired accuracy tolerance $\epsilon$,  sufficient large $l < min(m,n)$, block size $b$ and $v \geq 2$.
\ENSURE rank $k$, $\mathbf{V}(:,1:k)$and $\mathbf{G}(:,1:k)$.
\STATE $\mathbf{L} = []$, $\mathbf{U} = []$, $\mathbf{A}^{(0)} = \mathbf{A}$, $E =\|A\|_{F}^{2}$ and $acc =  \epsilon^{2}E$;
\STATE Passing matrix $\mathbf{A}$, $l$, and $v$ into Algorithm \ref{alg:gpowerlu} to get the orthonormal basis $\mathbf{V}$; \label{adaptive:step2}
\STATE $\mathbf{G} = \mathbf{A}\mathbf{V}$;  \COMMENT{$\mathcal{C}_{mm}mnl$} 
\FOR{$i = 1,2, 3, \cdots $}
   \STATE Let $t_{1} = (i-1)b+1$ and $t_{2} = ib$;
    \STATE $E = E - \|\mathbf{G}(:, t_{1}: t_{2})\|_{F}^{2}$; \label{adaptive:step7}
    \IF{$E\leq acc$} \label{adaptive:step8}
        \STATE STOP;
    \ENDIF \label{adaptive:step10}
\ENDFOR
 \end{algorithmic}
\caption{Effective Blocked Adaptive Rank Determination Algorithm}
\label{alg:adaptivewithoutupdating}
\end{algorithm}
\begin{remark}
Algorithm \ref{alg:adaptivewithoutupdating} is very efficient with the most costly step being \ref{adaptive:step2}, which is actually Algorithm \ref{alg:gpowerlu}.
If $\mathbf{V}\in \R^{n \times l}$ obtained from step \ref{adaptive:step2} has failed to achieve the given tolerance, then we need to regenerate the random Gaussian matrix and rerun the algorithm to obtain additional rank information from the remainder matrix, $\mathbf{A}-\mathbf{A}\mathbf{V}\mathbf{V}^{T}$. Although extra computation might be needed, this rank determination method is still efficient, since it requires fewer passes of $\mathbf{A}$ or it's remainder compared with the blocked random QB algorithm \cite{martinsson2016randomized}. 
\end{remark}

\subsection{PowerLU\_FP: A Method for Fixed Precision Low-Rank Matrix Approximation} So far we discussed a new blocked adaptive rank determination  Algorithm \ref{alg:adaptivewithoutupdating}. It produces a randomized LU algorithm suitable for the fixed precision low-rank approximation problem. We call it PowerLU\_FP; FP for fixed precision.  PowerLU\_FP is shown in Algorithm \ref{alg:powerlufp}. 

\begin{algorithm}[htb] 
\begin{algorithmic}[1]
\REQUIRE Given $\mathbf{A} \in \mathbb{R}^{m\times n}$, desired accuracy tolerance $\epsilon$,  sufficiently large $l < min(m,n)$, block size $b$ and $v \geq 2$.
\ENSURE rank $k$, $\mathbf{P}, \mathbf{Q}, \mathbf{L}, \mathbf{U}$ such that $\mathbf{PAQ} \approx \mathbf{LU}$, where $\mathbf{P}, \mathbf{Q}$ are orthogonal permutation matrices, $\mathbf{L} \in \mathbb{R}^{m \times k}$ and $\mathbf{U} \in \mathbb{R}^{k\times n}$ are the lower and upper triangular matrix, respectively..
\STATE Use Algorithm \ref{alg:adaptivewithoutupdating} to take $\mathbf{A}$, $\epsilon$, $b$, $l$, and $v$ to obtain $k$, $\mathbf{V}$, and $\mathbf{G}$; \label{powerlufp:step1}
\STATE  $[\mathbf{L}_{1}, \mathbf{U}_{1}, \mathbf{P}] = lu(\mathbf{G})$;\label{powerlufp:step2}\COMMENT{$\mathcal{C}_{lu} mk^{2}$}
\STATE  $\mathbf{B} = \mathbf{U}_{1}\mathbf{V}^{T} $; \COMMENT{$\mathcal{C}_{mm} nk^{2}$}
\STATE $[\mathbf{L}_{2}, \mathbf{U}_{2}, \mathbf{Q}] = lu(\mathbf{B}^{T})$; \COMMENT{$\mathcal{C}_{lu} nk^{2}$}
\STATE $\mathbf{L} = \mathbf{L}_{1}\mathbf{U}_{2}^{T}$; \COMMENT{$\mathcal{C}_{mm} mk^{2}$}
\STATE $\mathbf{U} = \mathbf{L}_{2}^{T}$; \label{powerlufp:step6}
\caption{PowerLU\_FP: PowerLU for the Fixed Precision Problem}
\label{alg:powerlufp}
\end{algorithmic}
\end{algorithm}

\begin{remark}
The time complexity of PowerLU\_FP is related to $l$. If $l \approx k$, then the time complexity of PowerLU\_FP is almost the same as PowerLU. Thus we can make PowerLU\_FP more efficient with a suitable $l$. 

\end{remark}

\section{Single-Pass Randomized LU for the Low-Rank Matrix Approximation} \label{sec:sec5} 
In this section, we will discuss the single-pass algorithm.  In \cite{halko2011finding}, the authors present a single-pass randomized QB algorithm, which will compress the input matrix $\mathbf{A}$ by using two randomized Gaussian matrices.  This algorithm needs to solve a system of linear equations, but is less accurate in practice \cite{halko2011finding, yu2017single}.  A more general single-pass algorithm can be found in \cite{tropp2017practical}. In \cite{yu2017single, yu2018efficient}, the authors proposed a single-pass scheme for principal component analysis, which is quite accurate when the matrix is stored in column-major or row-major format. Inspired by this last single-pass algorithms, we propose the single-pass randomized LU for the low-rank matrix approximation. Our proposed single-pass algorithm requires that the matrix is stored in column-major format. For a matrix in row-major format, a slight modification allows the algorithm to work.  The proposed algorithm is shown in the Algorithm \ref{alg:singlepasslu}.  

Given  $\mathbf{A} \in \mathbb{R}^{m\times n}$, suppose we want to find it's low-rank approximation with target rank $k$.  First we generate a random Gaussian matrix $\mathbf{\Omega} \in \mathbb{R}^{n \times k}$; then we compute 
\begin{equation}
    \mathbf{G} = \mathbf{A^{T}\Omega},
    \label{eq:G}
\end{equation}
and
\begin{equation}
    \mathbf{H} =\mathbf{AG}.
    \label{eq:H}
\end{equation}
Then we perform partial pivoting LU on $\mathbf{H}$ to define $\mathbf{L}_{1}$ and $\mathbf{U}_{1}$ with the permutation matrix, $\mathbf{P}$, encoding the pivots
    \begin{equation}
            \mathbf{P}\mathbf{H} = \mathbf{P}\mathbf{AG} = \mathbf{L}_{1}\mathbf{U}_{1}.
    \end{equation}
Taking the transpose gives us
\begin{equation}
    \mathbf{G}^{T}\mathbf{A}^{T}\mathbf{P} = \mathbf{U}_{1}^{T}\mathbf{L}_{1}^{T}. 
\end{equation}
Now multiply both sides by $\mathbf{G}^{T\dagger}$
\begin{equation}
    \mathbf{A}^{T}\mathbf{P} = \mathbf{G}^{T\dagger}\mathbf{U}_{1}^{T}\mathbf{L}_{1}^{T}.
\end{equation}
We now perform partial pivoting LU on $\mathbf{G}^{T\dagger}\mathbf{U}_{1}^{T}$ to obtain $\mathbf{L}_{2}$ and $\mathbf{U}_{2}$ with the permutation matrix, $\mathbf{Q}$, encoding the pivots
\begin{equation}
    \mathbf{Q}\mathbf{G}^{\dagger T}\mathbf{U}_{1}^{T} = \mathbf{L}_{2}\mathbf{U}_{2}.
\end{equation}
We have now computed the desired lower and upper triangular matrices as
\begin{equation}
    \mathbf{L} = \mathbf{L}_{1}\mathbf{U}_{2}^{T},
\end{equation}
and
\begin{equation}
    \mathbf{U} = \mathbf{L}_{2}^{T}.
\end{equation}

\begin{remark}
In practice, if a matrix is stored in column-major format, we compute step (\ref{eq:G}) and (\ref{eq:H}) with one pass over the matrix $\mathbf{A}$ as follows \cite{yu2018efficient}:

Let $\mathbf{A}(:, i)$ be the $i$th column of $\mathbf{A}$, then the $i$th row of $\mathbf{G}(i,:)$ is computed by
\begin{equation}
    \mathbf{G}(i,:) = \mathbf{A}(:,i)^{T} \mathbf{\Omega}.
    \label{eq:gi}
\end{equation}  
By using $\mathbf{A}(:, i)$ and $\mathbf{G}(i,:)$, we compute $\mathbf{H}_{i}$ as an outer product
\begin{equation}
    \mathbf{H}_{i} = \mathbf{A}(:,i)\mathbf{G}(i,:).  
    \label{eq:hi}
\end{equation}
We only to access the $i$th column of $\mathbf{A}$ one time to compute (\ref{eq:gi}) and (\ref{eq:hi}).
Finally we obtain $\mathbf{G}$ and assemble $\mathbf{H} = \sum_{i}\mathbf{H}_{i}$ by accessing  $\mathbf{A}$ once per successive column access. 

\end{remark}
\begin{algorithm}[htb] 
\begin{algorithmic}[1]
\REQUIRE Given $\mathbf{A} \in \mathbb{R}^{m\times n}$, desired rank  $k$. 
\ENSURE Matrix: $\mathbf{P}, \mathbf{Q}, \mathbf{L}, \mathbf{U}$ such that $\mathbf{PAQ} \approx \mathbf{LU}$, where $\mathbf{P}, \mathbf{Q}$ are orthogonal permutation matrices, $\mathbf{L} \in \mathbb{R}^{m \times k}$ and $\mathbf{U} \in \mathbb{R}^{k\times n}$ are lower and upper triangular matrices.
\STATE Generate a randomized Gaussian matrix $\mathbf{\Omega}$ with size $m \times k$; \label{singlepass:step1}
\STATE Compute $\mathbf{G} = \mathbf{A}^{T}\mathbf{\Omega}$; \label{singlepass:step2}
\STATE Compute $\mathbf{H} = \mathbf{A} \mathbf{G}$; \label{singlepass:step3}
\STATE  $[\mathbf{L}_{1}, \mathbf{U}_{1}, \mathbf{P}] = lu(\mathbf{H})$;
\STATE  $[\mathbf{L}_{2},\mathbf{U}_{2},\mathbf{Q}] = lu(\mathbf{G}^{\dagger T}\mathbf{U}_{1}^{T})$;
\STATE $\mathbf{L} = \mathbf{L}_{1}\mathbf{U}_{2}^{T}$; 
\STATE $\mathbf{U} = \mathbf{L}^{T}$;
\end{algorithmic}
\caption{Single-Pass Randomized LU Factorization}
\label{alg:singlepasslu}
\end{algorithm}
In Algorithm \ref{alg:singlepasslu}, we do not use the oversampling parameter for the random Gaussian matrix. If oversampling is used, we need to do the truncation to get the desired results for expected rank $k$.  Experiments show that our proposed algorithm can achieve results almost as accurate as the method in \cite{yu2017single}, and produce better results than the single-pass algorithm in \cite{halko2011finding}. 

\section{Numerical Results}\label{sec:sec6}
In this section, we will present some numerical experiments to show the efficiency and accuracy of our proposed methods.  The experiments were carried out on a machine with
an Intel(R) Xeon(R) E5-2603 v4 @ 1.70GHz 6 core CPU  with 16 GB of RAM.  Our code was implemented in MATLAB.  Since these algorithms involve randomness, all the results shown in the this section are the average of 20 random calculations. The numerical test settings are similar to those in \cite{yu2018efficient}. The algorithms used in the this section are the following: 
\begin{itemize}
    \item PowerLU: The generalized PowerLU, where the power iteration is computed by the Algorithm \ref{alg:gpowerlu}. 
    \item PowerLU\_FP: Algorithm \ref{alg:powerlufp}.
    \item SinglePass: Algorithm \ref{alg:singlepasslu}.
    \item RandSVD: RandSVD Algorithm \ref{alg:rsvd} , where the power iteration is computed by the Algorithm \ref{alg:basicpoweriterationsvd} .
    \item RandLU: RandLU Algorithm \ref{alg:rlu} , where the power iteration is computed by the Algorithm \ref{alg:basicpoweriteration}.
    \item RandLU\_Original: RandLU Algorithm \ref{alg:rlu}, where the power iteration is computed directly without reorthogonalization.  
    \item RangeFinder: Algorithm 4.2 in  \cite{halko2011finding}.
    \item RandQB\_b: Blocked random QB Algorithm in \cite{martinsson2016randomized}.
    \item RandQB\_FP: Algorithm 4 in \cite{yu2018efficient}. Note that here FP means Few Passes in the RandQB\_FP algorithm as opposed to our meaning of fixed precision.
    \item SinglePass2011: Single-Pass Algorithm in \cite{halko2011finding}.
\end{itemize}

In this paper, $v$ denotes the number of passes over the matrix $\mathbf{A}$ for the algorithms: PowerLU and PowerLU\_FP, and $p$ denotes the power iteration exponent for the algorithms: RandLU, RandSVD, RandQB\_b, and RandQB\_FP. However, we have the following relation between $v$ and $p$:
\begin{equation}
v = 2p + 2.
\label{eq:rela}
\end{equation}
Therefore,  for simplicity, we only use $p$.  The number of passes over the matrix will be computed by equation (\ref{eq:rela}) for the PowerLU and PowerLU\_FP algorithms.

\subsection{Comparison of Execution Time} In this section, we will examine the execution time of our proposed algorithms.  To test our algorithms, we generate random square matrices with variety of sizes ranging from 2000 to 32000. 

In the first experiment, we set $l  = 200$ for all matrices regardless of their size.  All the experiments are done both with and without power iteration. The results are shown in the Fig.~\ref{fig:1}. For RandQB\_b and RandLU results are not available when $n=32000$ due to machine memory constraints.   Fig.~\ref{fig:1} on the left is the result without the power iteration.  We can see that PowerLU achieves up to 1.80X speedup over RandLU, and has almost the same performance as RandSVD.  The results with the power iteration, where $p = 1$, are shown on the right. The gap between RandLU and RandSVD is smaller since more time is spent by both on matrix-matrix multiplication. PowerLU gets a 1.44X speedup over RandLU. 
We notice that when $p = 0$, RandLU is actually the original randomized LU algorithm, without power iteration. However, our tests show that RandLU is slower than RandSVD, but our PowerLU algorithm provides results comparable with RandSVD.  

For algorithms used in the fixed precision problem, we set the block size $b = 20$.  PowerLU\_FP obtains up to a 4.83X speedup without power iteration, and up to a 3.41X speedup when with power iteration parameter $p = 1$ compared with RandQB\_b.  From Fig.~\ref{fig:1}, we see our proposed PowerLU\_FP has a 1.16X speedup over RandQB\_FP when $n = 2000$ without power iteration.  When power iteration is used,  PowerLU\_FP  obtains up to a 1.21X speedup when $n = 2000$ over RandQB\_FP. 
We can see that as the matrix size increases, the advantage over RandQB\_FP becomes smaller.  RandQB\_FP is currently the state-of-art algorithm for the fixed precision low-rank approximation using QR factorization, and it produces an orthonormal basis for $\mathbf{A}$ satisfying (\ref{eq:adaperror}). However, RandQB\_FP does not actually produce a factorization, and there are extra steps needed to produce one, such as the SVD.  Our PowerLU\_FP does produce a low rank LU factorization. 

\begin{figure}[ht]
\centering
\includegraphics[width=0.98\textwidth]{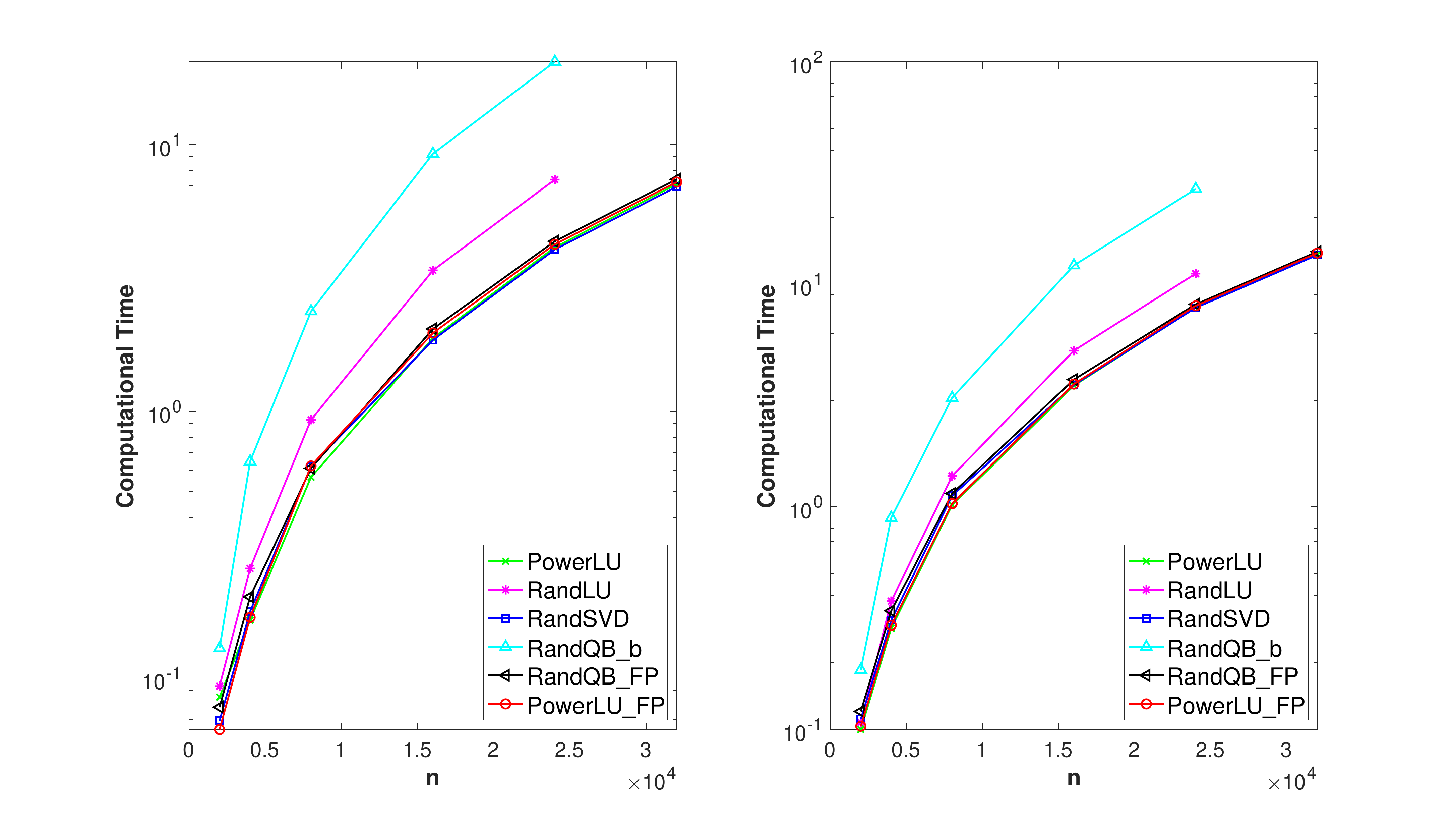}
\caption{Runtime of the algorithms with fixed $l = 200$ for dense matrices, without power iteration (left), and with power iteration (right).}
\label{fig:1}
\end{figure}

In the second experiment, we test the efficiency of our proposed algorithms on sparse matrices. We generate our test matrices using the MATLAB command $sprand$ with density $0.3\%$.  The rest of the  parameters are set to the values used above. 
The results are shown in Fig.~\ref{fig:1_2}.  Again, results for RandLU are not available when $n=32000$.  As before, the left and right sides of the Fig.~\ref{fig:1_2} are the results without and with power iteration respectively.  When $n < 20000$, RandLU achieves the best results. This is because a sparse matrix, matrix-matrix multiplication is faster than for a dense matrix, and because
PowerLU must execute an extra QR decomposition to obtain the orthonormal basis.
 However, PowerLU performs slightly better than RandSVD. When $n \geq 20000$, we see that RandLU, PowerLU and RandSVD have comparable running times.  PowerLU\_FP can has up to an 11.36X speedup over RandQB\_b, and  is generally faster than RandQB\_FP for sparse matrices. 

\begin{figure}[ht]
\centering
\includegraphics[width=0.98\textwidth]{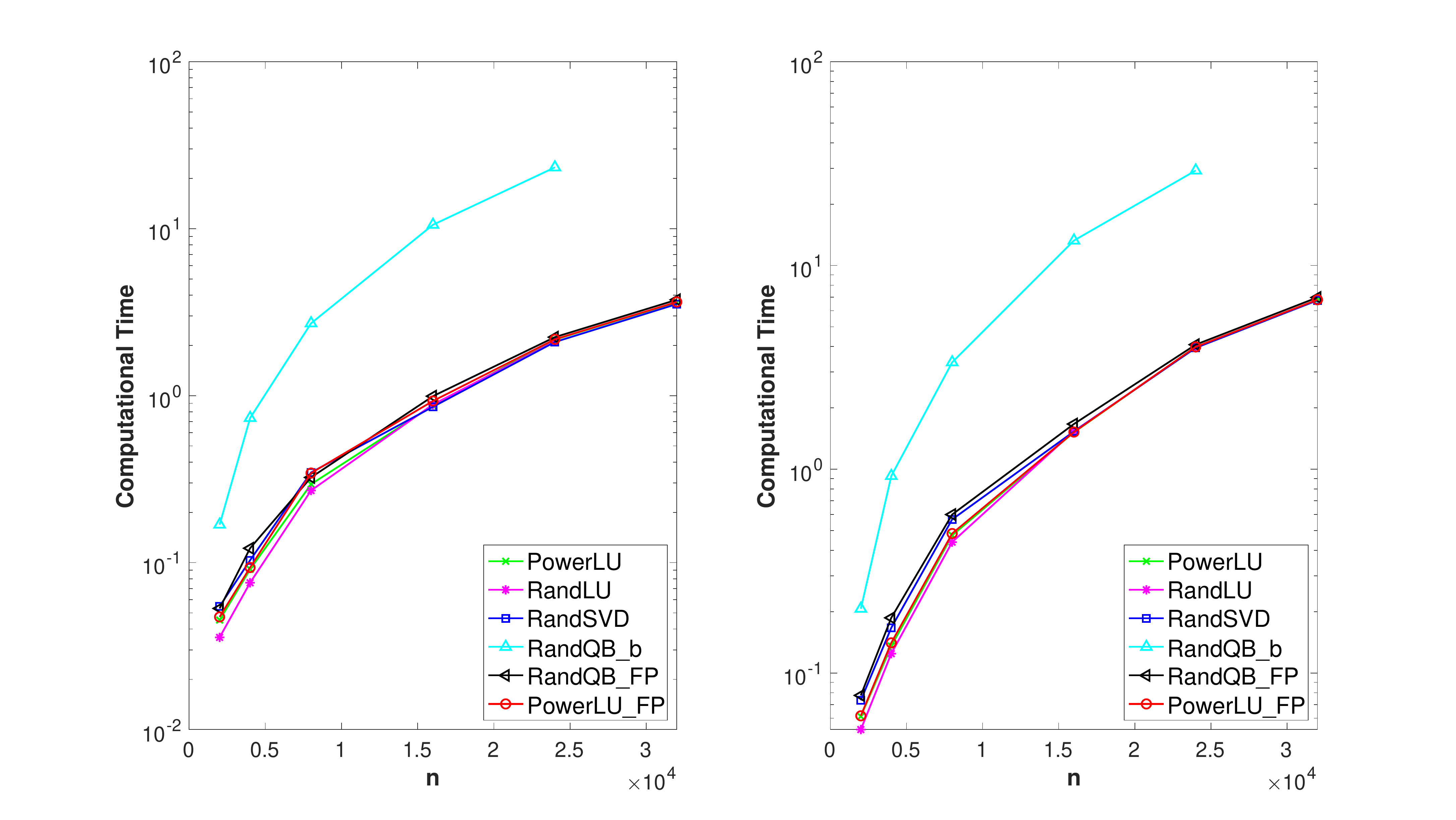}
\caption{Runtime of the algorithms with fixed $l = 200$ for sparse matrices, without power iteration (left), and with power iteration (right).}
\label{fig:1_2}
\end{figure}

In the third experiment, we fixed the size of the matrix and varied of the rank, $k$, from $100$ to $1000$. The results are shown in Fig.~\ref{fig:1_3}.  PowerLU\_FP shows up to a 4.21X speedup over randQB\_b without  power iteration. However, the speedup decreases a little when we set $p = 1$ in the power iteration, since more time will be spent on matrix-matrix multiplication.  From Fig.~\ref{fig:1_3}, we can see that our PowerLU\_FP generally outperforms RandQB\_FP for different values of the rank, $k$, especially when $k$ is large. In terms of the fixed rank algorithms, PowerLU always outperforms RandLU and RandSVD. However, RandLU is actually slower than RandSVD with a relatively small value of $k$ when compared with RandSVD. 

\begin{figure}[ht]
\centering
\includegraphics[width=0.98\textwidth]{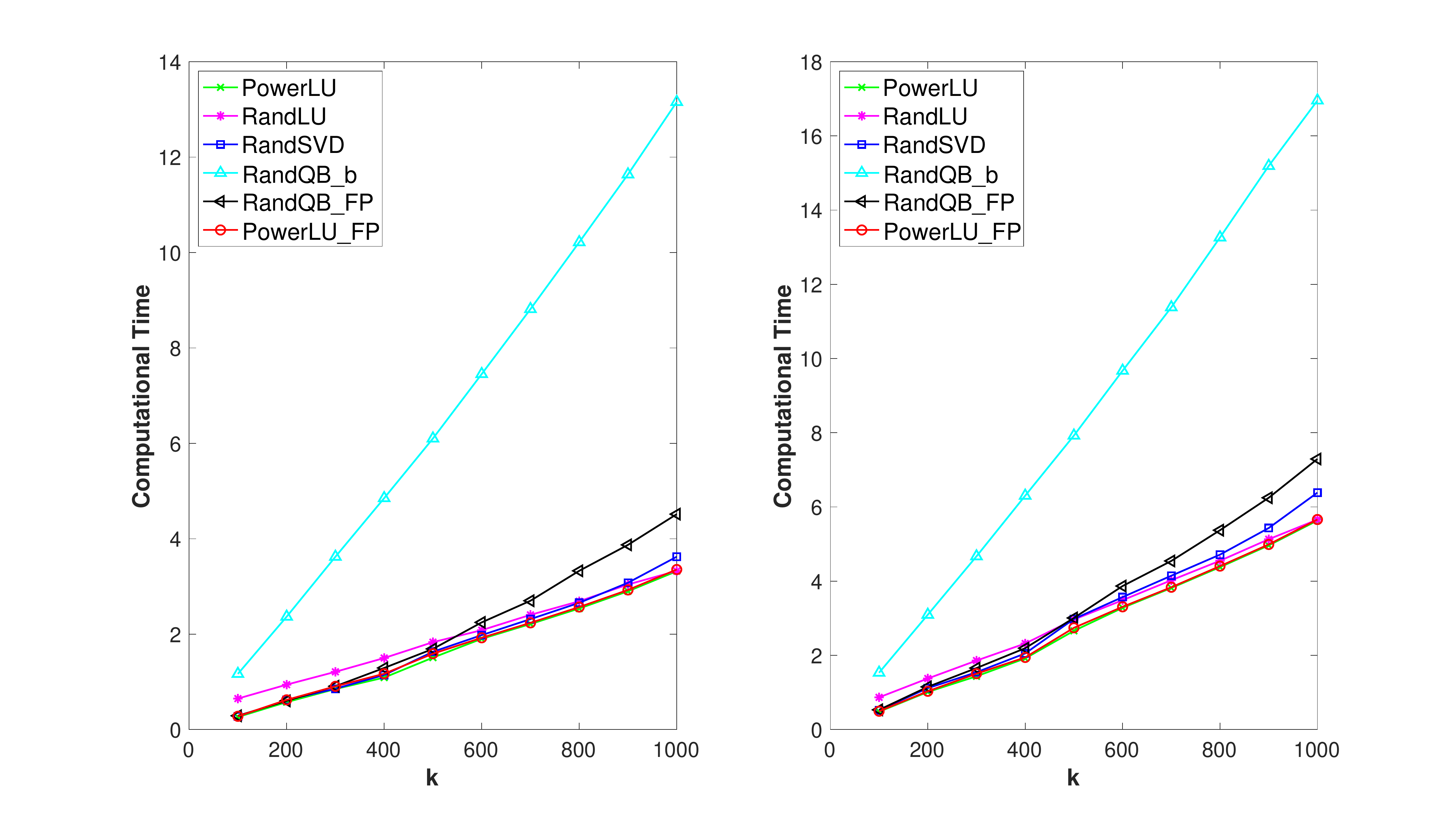}
\caption{Runtime of the algorithms with fixed matrix size $n = 8000$, without power iteration (left), and with power iteration (right).}
\label{fig:1_3}
\end{figure}
 
\subsection{Comparison of Accuracy}
\label{sec:n2}
To test the accuracy of our methods, we randomly generate three types of the matrices based on the singular value decay speed as follows \cite{gopal2018powerurv, yu2018efficient}:
\begin{itemize}
    \item \textit{Matrix Type 1:(Slow decay)}: This is a matrix $\mathbf{A = U\Sigma V}^{T} $, where $\mathbf{U}$ and $\mathbf{V}$ are
random matrices with  orthonormal columns, and $\mathbf{\Sigma}$ is a rectangular, diagonal matrix with entries $\mathbf{\Sigma}(k, k) = \frac{1}{k^{2}}$.
 \item \textit{Matrix Type 2:(Fast decay)}: $\mathbf{\Sigma}(k, k) = e^{-k/7}$.
 \item \textit{Matrix Type 3:(S-Shaped decay)}: $\mathbf{\Sigma}(k, k) = 0.0001 + (1 + e^{k-30})^{-1}$.

\end{itemize}

Suppose $\mathbf{A}$ is the original matrix and $\mathbf{A}_{k}$ is the approximated matrix obtained by using the randomized algorithm.  Then we can compute the relative Frobenius error as following:
\begin{equation}
    \label{eq:error}
   \frac{\|\mathbf{A} - \mathbf{A}_{k}\|_{F}}{\|\mathbf{A}\|_{F}}.
\end{equation} 

For each matrix type, we generate a square matrix of size 2000, for which we compare the errors of our proposed algorithms with truncated SVD, RandSVD, and RandLU for different values of $l$.  We will also set different power iteration parameter values, $p$. For PowerLU and PowerLU\_FP, we need to use equation (\ref{eq:rela}) to compute the actually number of matrix passes.
As shown in the Figs.~\ref{fig:2}, \ref{fig:3} and \ref{fig:4}, our proposed PowerLU and PowerLU\_FP have accuracy comparable to RandSVD with power iteration.  For our proposed algorithms, when $p = 1$, they achieves almost the same accuracy as when $p =2$. RandLU\_Original, without reorthogonalization, loses some accuracy as $l$ grows.  For our proposed  algorithm, any number of passes $v \geq 2$ of $\mathbf{A}$ is possible as opposed an even number of passes.  For example, If PowerLU uses 3 passes of $\mathbf{A}$ to achieve the required accuracy, RandSVD and RandLU would need 4 passes since these two algorithms can only accept even number of passes.  
\begin{figure}[ht]
\centering
\includegraphics[width=0.98\textwidth]{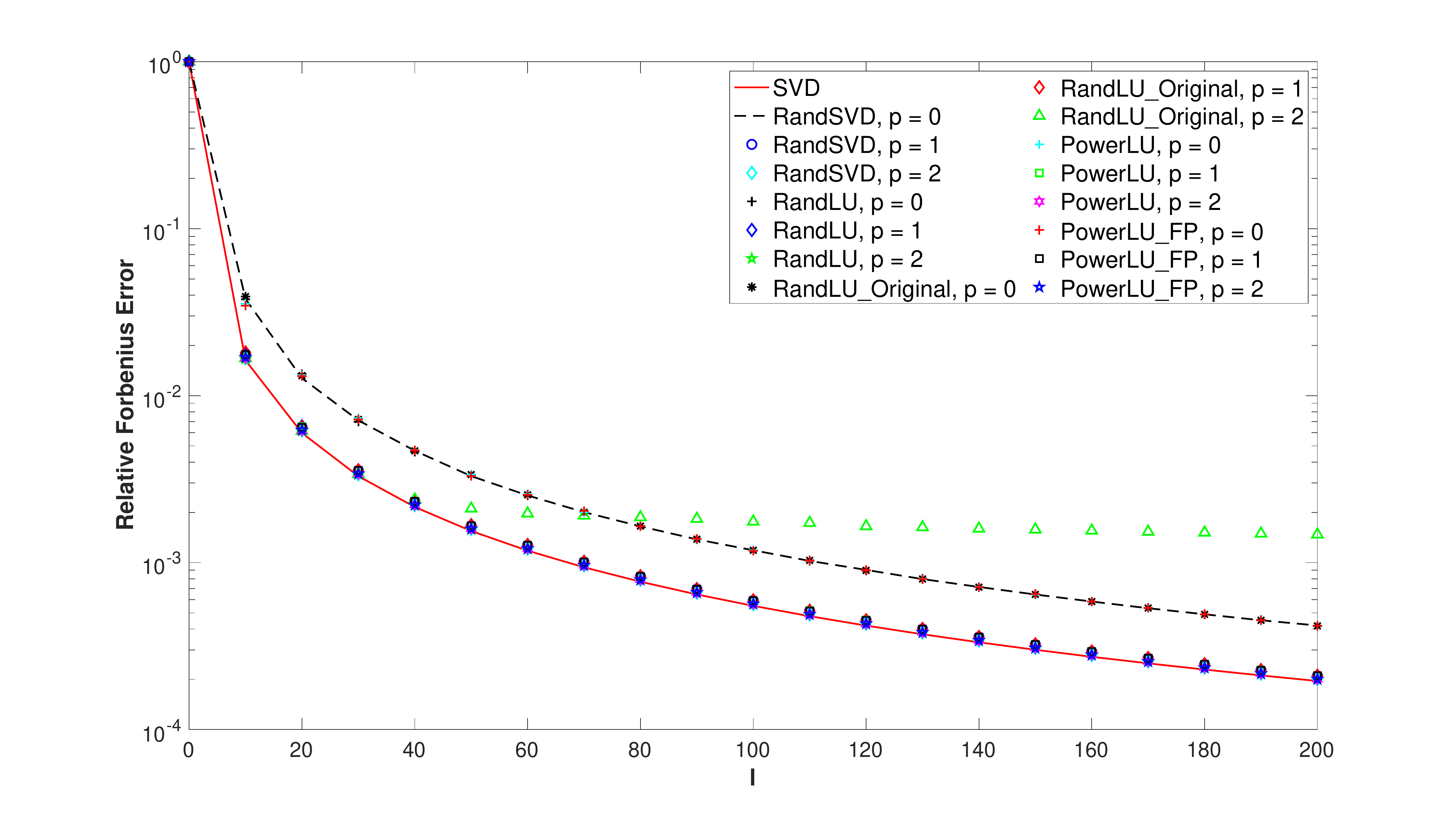}
\caption{Errors for Matrix Type 1 (Slow decay)}
\label{fig:2}
\end{figure}
\begin{figure}[ht]
\centering
\includegraphics[width=0.98\textwidth]{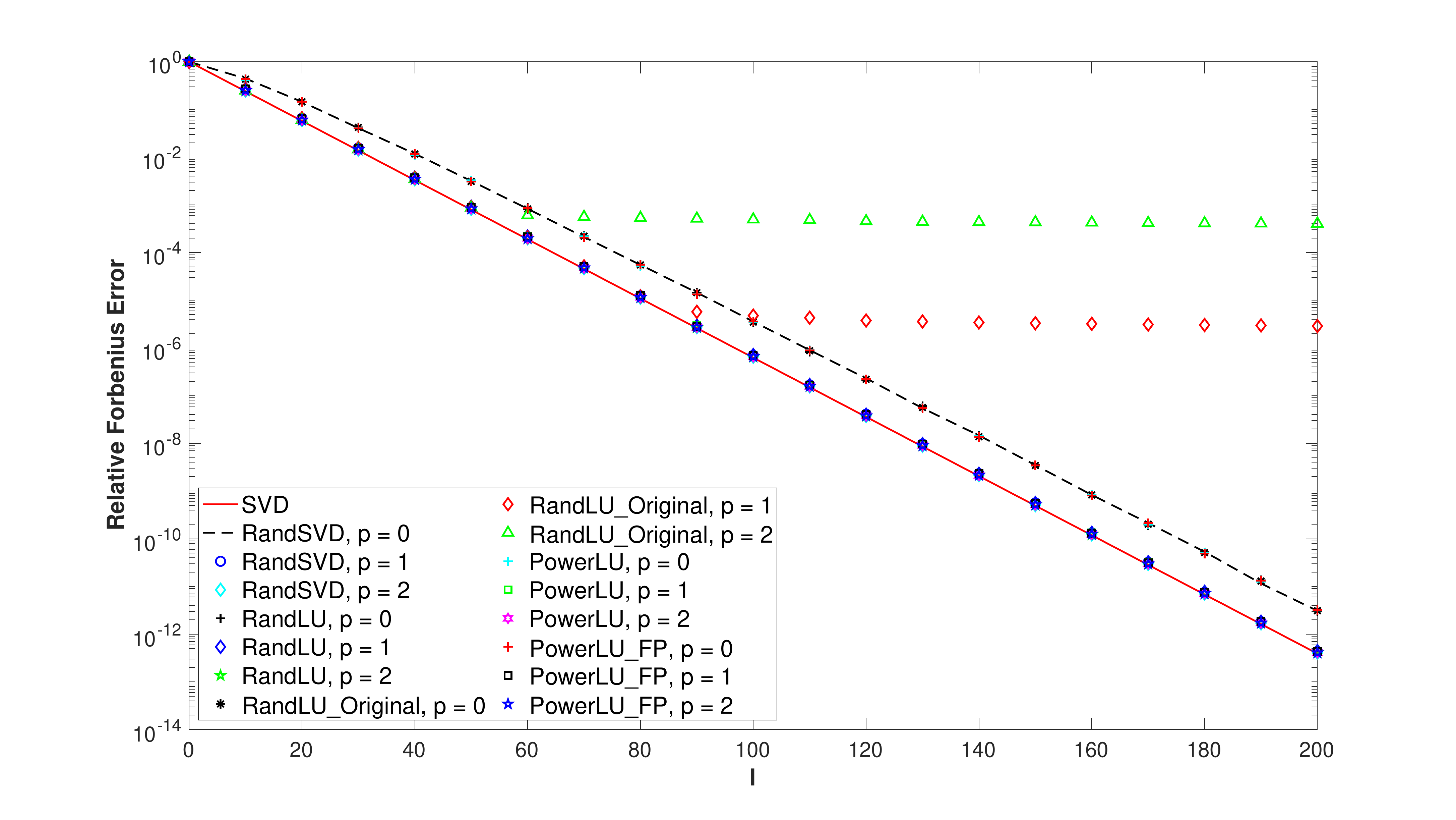}
\caption{Errors for Matrix Type 2 (Fast decay)}
\label{fig:3}
\end{figure}
\begin{figure}[ht]
\centering
\includegraphics[width=0.98\textwidth]{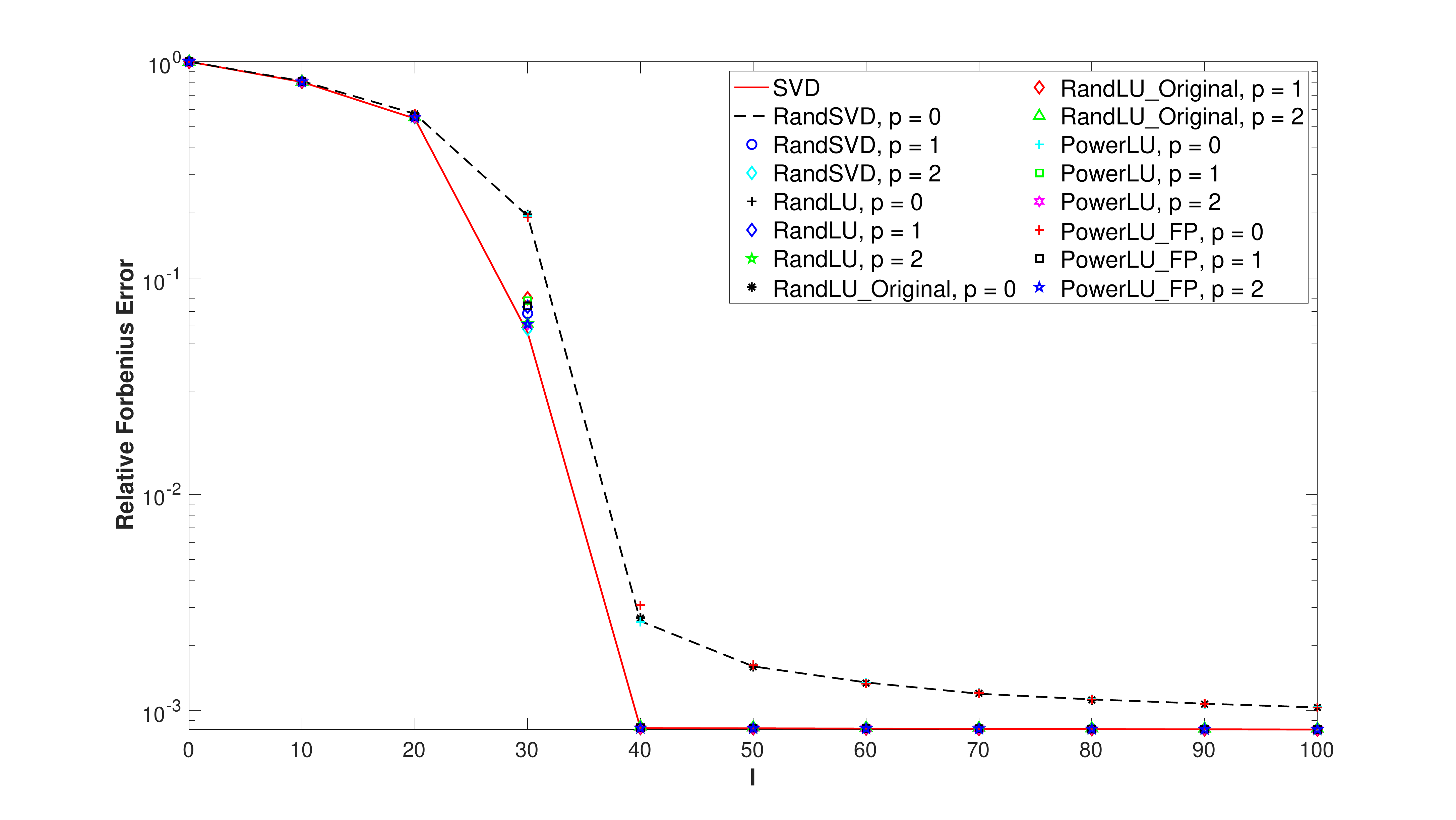}
\caption{Errors for Matrix Type3 (S-shaped decay)}
\label{fig:4}
\end{figure}

\subsection{Performance of the Single-Pass Algorithm} In this section, we measure the accuracy of our proposed LU-based single-pass low-rank approximation Algorithm \ref{alg:singlepasslu}.  The results are shown in  Fig.~\ref{fig:7}.  We can see that our proposed algorithm has almost the same accuracy as RandQB\_FP \cite{yu2018efficient}, and both of them are better than SinglePass11\cite{halko2011finding}.  We only show results for Matrix Types 1 and 2.
\begin{figure}[ht]
\centering
\includegraphics[width=0.98\textwidth]{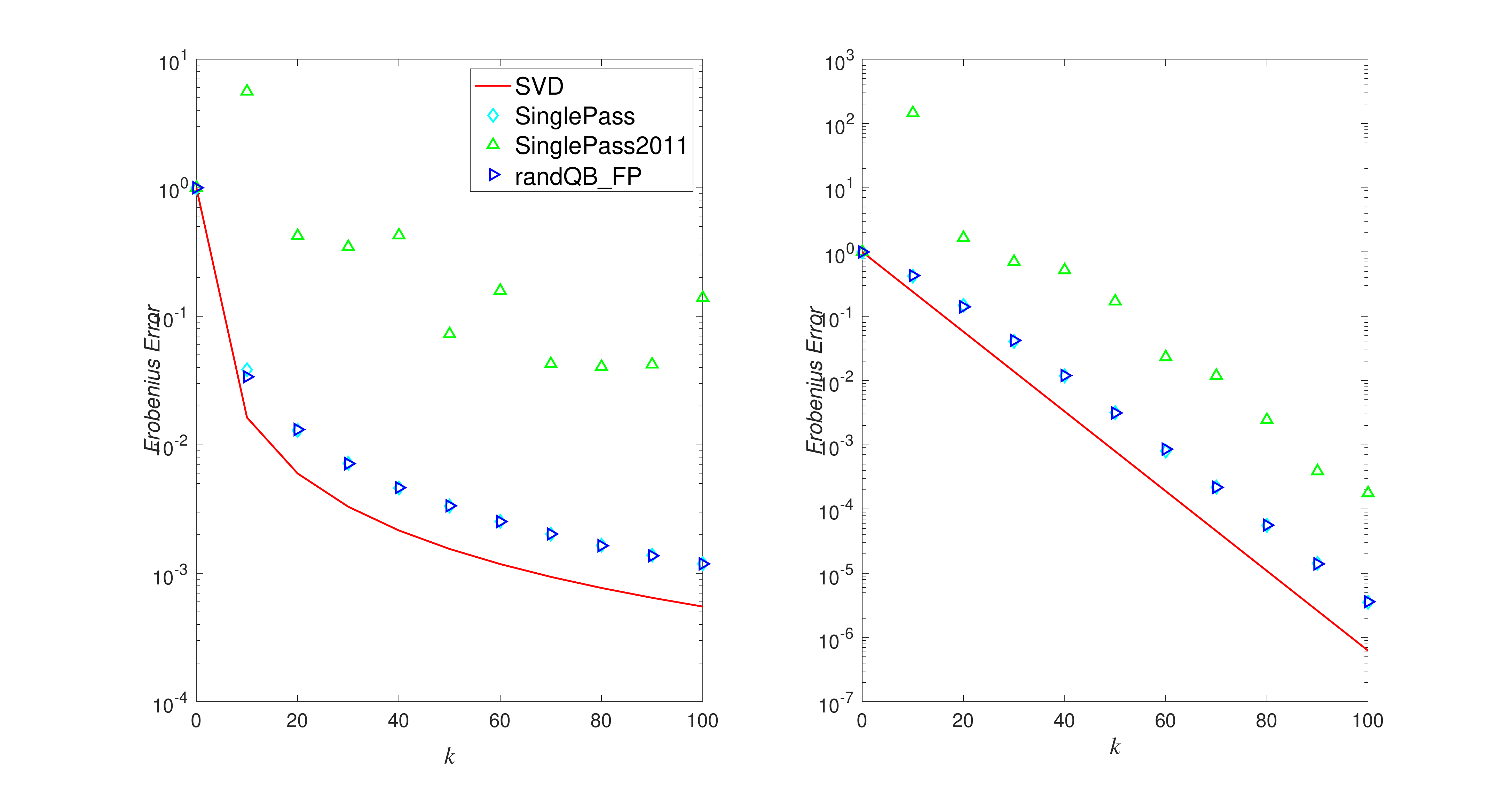}
\caption{Approximation errors of the single-pass algorithms and truncated SVD for Matrix Type 1 (left) and Matrix Type 2 (right).}
\label{fig:7}
\end{figure}

\subsection{Results of Fixed Precision Approximation}
In this section, we will present the results of our proposed algorithms used to solve the fixed precision low-rank matrix approximation problem.  The Eckart and Young Theorem \cite{eckart1936approximation} implies that the optimal solution can be achieved via truncated SVD.  For the optimal solution, we could compute the SVD and then calculate the error as $(\sum_{i = k +1}^{min(m,n)}\sigma_{i}^{2})^{1/2}$, where $k$ is the rank that satisfies the given tolerance. 
Our proposed method, PowerLU\_FP, will be compared with truncated SVD, RandQB\_FP and RangeFinder. For both PowerLU\_FP and RandQB\_FP, we set $l = 50b$ sufficiently large, where $b$ is the block size.
PowerLU\_FP can only generate a matrix of rank that is a multiple of the block size $b$.  To produce the rank more precisely, we can replace steps \ref{adaptive:step8}-\ref{adaptive:step10} in Algorithm \ref{alg:adaptivewithoutupdating} by the following steps \ref{peci:step1} - \ref{peci:step2} . 

\begin{algorithm}[htb] 
\begin{algorithmic}[1]
\setcounterref{ALC@line}{adaptive:step7}
\FOR{$j = 1, 2, \cdots, b$} \label{peci:step1}
\STATE $E = E - \|\mathbf{G}(:,t1+j)\|_{F}^{2}$;
\IF{$E < err$}
\STATE break; 
\ENDIF
\ENDFOR \label{peci:step2}
\end{algorithmic}
\label{alg:preciselyerror}
\end{algorithm}

First, we generate the three types of square matrices as in \S \ref{sec:n2}, with size 8000. For PowerLU\_FP and RandQB\_FP, we set $p = 1$ (corresponding to $v = 4$ for PowerLU\_FP), and we set $b = 10$ for both algorithms except for the last case where we set $b = 40$ which is needed to produce a sufficiently large $l$.  We test our algorithm with different$\epsilon$ values. The results are listed in Table \ref{table:1}.  We can see that our proposed PowerLU\_FP  produces a rank close to the truncated SVD but with a speedup greater than 10X. 

We next test our algorithm with a real data, which is the standard image \cite{yu2018efficient} shown in Fig.~\ref{fig:8}
represented by a $9504 \times 4752$ matrix.  We can see from Table \ref{table:2}, when we set $p = 2$ (corresponding $v = 6$ for PowerLU\_FP), the  rank is reduced significantly.  With the same  exponent, the results are almost the same with PowerLU\_FP despite a different block size.  Our proposed algorithm nearly achieves a 7X size reduction from the original, which is comparable to RandQB\_FP.  The compressed image obtained from PowerLU\_FP is shown in Fig.~\ref{fig:9}.  

\begin{table}[htb]
\centering
\caption{Results for fixed precision problems.}
\label{table:1}
\resizebox{\columnwidth}{!}{
\begin{tabular}{*{10}{c}}
\toprule
 Matrix &\multirow{2}*{$\mathbf{\epsilon}$} & \multicolumn{2}{c}{PowerLU\_FP}& \multicolumn{2}{c}{RandQB\_FP} & \multicolumn{2}{c}{Truncated SVD} & \multicolumn{2}{c}{Range Finder }\\
 \cmidrule{3-4} \cmidrule{5-6} \cmidrule{7-8} \cmidrule{9-10}
 Type & &Rank & Time(s) & Rank & Time(s) & Rank & Time(s) & Rank & Time(s)\\ \midrule

  \multirow{2}*{1}& 1e-2 & 15    & 2.511   & 16& 2.515 & 15& \multirow{2}*{99.83}& 116&  0.4271\\
  & 1e-4 &328  & 2.620 & 328  & 2.847 &313 & & 2101& 23.118\\ 
  \midrule
   \multirow{2}*{2}& 1e-4 & 66    & 2.515   & 66 & 2.462 & 65& \multirow{2}*{86.06} &99 & 0.3670\\
  & 1e-5 &82    &  2.505 & 82 & 2.468 & 81& &115 &0.4079\\ 
\midrule
 \multirow{2}*{3}& 1e-2 & 32   & 2.461   & 33 & 2.423 & 32&  \multirow{2}*{86.88} & 3635&61.477\\
  & 1.5e-3 &1588    &  11.24 & 1588 & 13.47 & 1587&&7925 & 272.62\\ 
  \bottomrule
\end{tabular}
}
\end{table}

\begin{table}[htb]
\centering
\caption{Results for fixed precision with real (image) data.}
\label{table:2}
\resizebox{\columnwidth}{!}{
\begin{tabular}{*{11}{c}}
\toprule
 \multirow{2}*{Data} &\multirow{2}*{$\mathbf{\epsilon}$} & \multirow{2}*{Parameters} &\multicolumn{2}{c}{PowerLU\_FP}& \multicolumn{2}{c}{RandQB\_FP} & \multicolumn{2}{c}{Truncated SVD} & \multicolumn{2}{c}{Range Finder }\\
 \cmidrule{4-5} \cmidrule{6-7} \cmidrule{8-9} \cmidrule{10-11}
  & & &  Rank & Time(s) & Rank & Time(s) & Rank & Time(s) & Rank & Time(s)\\ \midrule
  \multirow{4}*{Image}&\multirow{4}*{0.1}& $P = 1, b = 10$ & 472  & 1.967 & 471 & 2.415 & \multirow{4}*{426}& \multirow{4}*{40.62}&\multirow{4}*{2892} & \multirow{4}*{45.99}\\
  & &$P = 1, b = 20$ & 471 & 3.753  & 472 & 4.008 &  & & &\\ 
  \cmidrule{3-7}
  &  & $P = 2, b = 10$   & 443   & 2.749 & 443 & 3.310 &  & & &\\
  & &$P = 2, b = 20$   & 443 & 5.388 & 443 &6.122  & & & \\ 
\bottomrule
\end{tabular}
}
\end{table}

\begin{figure}[ht]
\centering
\includegraphics[width=0.98\textwidth]{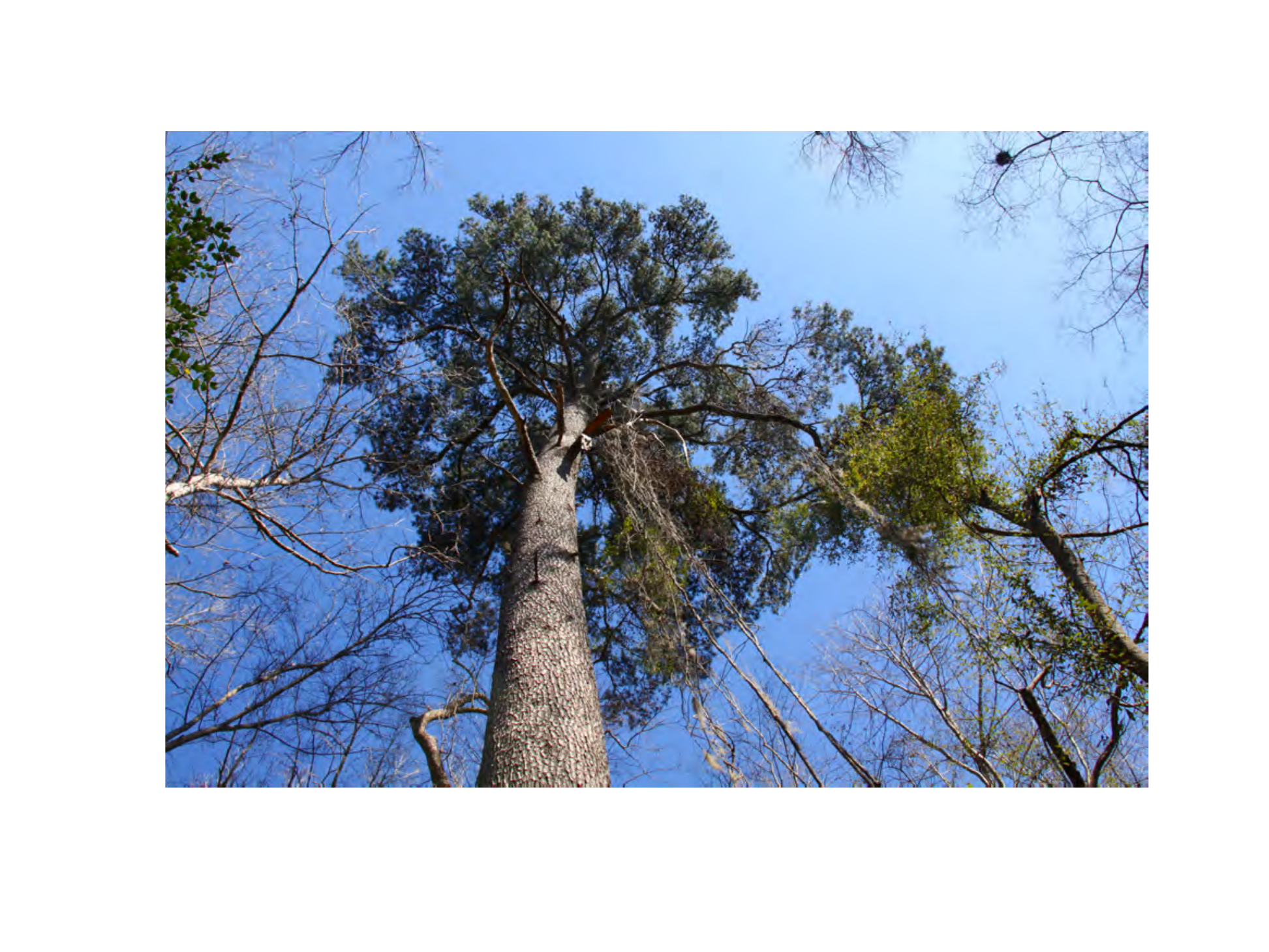}
\caption{ The original image}
\label{fig:8}
\end{figure}

\begin{figure}[ht]
\centering
\includegraphics[width=0.88\textwidth]{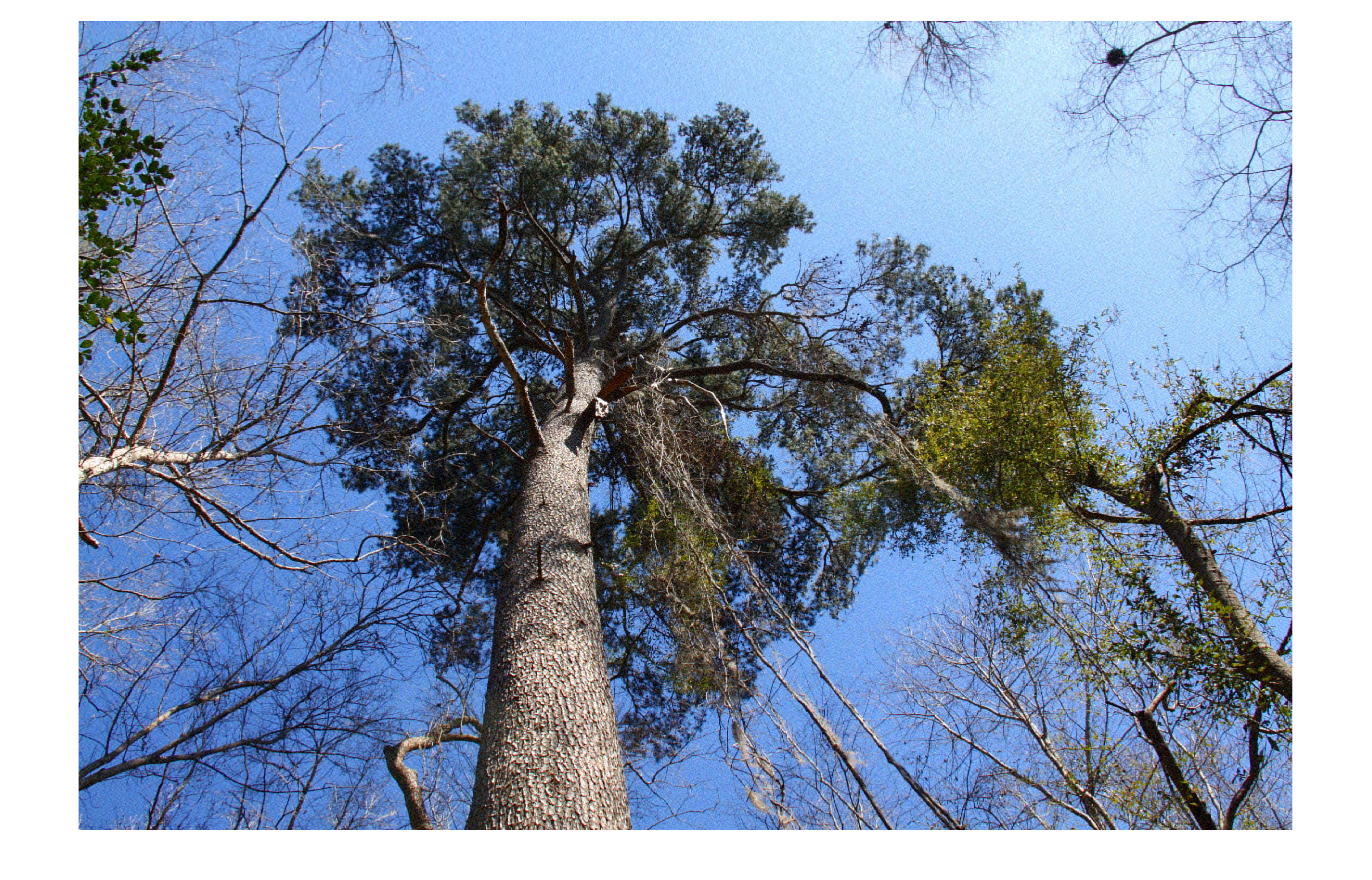}
\caption{ The compressed image obtained by PowerLU\_FP with $10\%$ error}
\label{fig:9}
\end{figure}
\section{Conclusions and Future Work}

In this manuscript, we discussed two novel randomized LU algorithms: PowerLU and PowerLU\_FP to solve both the fixed-rank and the fixed precision low-rank approximation problems. A randomized LU (RandLU) Algorithm \ref{alg:rlu} was first proposed in \cite{shabat2018randomized}. Our tests show that compared with RandSVD, RandLU is slower for large matrices.  Moreover, round-off errors in power iteration will extinguish the small singular values of the input matrix. For the fixed precision problem, RandLU will not work.  In this paper, we introduced a reorthogonalization Algorithm \ref{alg:basicpoweriteration} for the RandLU Algorithm \ref{alg:rlu}.  With our reorthogonalization procedure, RandLU can eliminate this round-off issue.  Compared with RandLU, our new randomized LU Algorithm \ref{alg:bpowerlu}, called PowerLU, is generally faster and accurate.  PowerLU is based on the orthonormal basis of the input matrix $\mathbf{A}^{T}$ and the approximation (\ref{eq:rightsingular}). PowerLU allows an arbitrary number of passes $v\geq 2$ of the matrix $\mathbf{A}$ via the generalized power iteration reorthogonalization Algorithm \ref{alg:powerlufp}. However, RandLU only allows even number of passes of $\mathbf{A}$.

To solve the fixed precision problem, we proposed an efficient blocked adaptive rank determination Algorithm \ref{alg:adaptivewithoutupdating}, which uses an efficient error indicator (\ref{eq:errorerror}) without the need for a matrix update. It can determine a rank close to the optimal rank produced by SVD.  Then we proposed PowerLU\_FP Algorithm \ref{alg:powerlufp}, which is based on Algorithm
\ref{alg:adaptivewithoutupdating}. This variant is faster than the randomized QB algorithm in \cite{martinsson2016randomized}.  We also proved the correctness of our proposed algorithms.
For the problem where accessing the matrix is expensive, we proposed a single-pass LU-based algorithm, which requires that matrix be stored in column-major format.  Tests establish the accuracy of our proposed single-pass Algorithm \ref{alg:singlepasslu}.

We plan to build on this work in several ways.  First, we plan to implement the algorithms in a high-level computer language such as C/C++ for the sake of efficiency and portability.  A longer term goal is to create a distributable piece of mathematical software that can form the basis of a library for randomized linear algebra.  We believe that such software will be of interest to may different communities, including big data and data science.  We also plan to implement our algorithms on GPUs.  We are particularly enthusiastic to explore this with our blocked algorithm.
One more potential research direction is the application of the proposed algorithms on image and video processing since these our single-pass version is a stream algorithm.  Finally, we plan on considering higher-order tensor decomposition including CANDECOMP/PARAFAC (CP) decomposition \cite{kiers2000towards, kolda2009tensor},  Tucker decomposition\cite{kolda2009tensor, tucker1963implications, tucker1966some, tucker1964extension} and Tensor-Train decomposition \cite{oseledets2011tensor}.

\bibliographystyle{siamplain}
\bibliography{references}

\end{document}

%% file: ex_shared.tex
\usepackage{lipsum}
\usepackage{amsfonts}
\usepackage{graphicx}
\usepackage{epstopdf}
\usepackage{algorithmic}
\ifpdf
  \DeclareGraphicsExtensions{.eps,.pdf,.png,.jpg}
\else
  \DeclareGraphicsExtensions{.eps}
\fi

\newsiamremark{remark}{Remark}
\newsiamremark{hypothesis}{Hypothesis}
\crefname{hypothesis}{Hypothesis}{Hypotheses}
\newsiamthm{claim}{Claim}

\headers{Randomized LU Matrix Factorization}{Bolong Zhang and Michael Mascagni}

\title{Pass-Efficient Randomized LU Algorithms for Computing Low-Rank Matrix Approximation
}

\author{Bolong Zhang\thanks{Department of Computer Science, Florida State University, FL, USA 
(\email{bzhang@cs.fsu.edu}).}
\and Michael Mascagni\thanks{Department of Computer Science, Florida State University, FL and Applied and Computational Mathematics Division, NIST, Gaithersburg, MD, USA (\email{mascagni@fsu.edu}).}
}

\usepackage{amsopn}

\makeatletter
\newcommand*{\addFileDependency}[1]{
  \typeout{(#1)}
  \@addtofilelist{#1}
  \IfFileExists{#1}{}{\typeout{No file #1.}}
}
\makeatother

\newcommand*{\myexternaldocument}[1]{%
    \externaldocument{#1}%
    \addFileDependency{#1.tex}%
    \addFileDependency{#1.aux}%
}

\usepackage{algorithmic,eqparbox,array} 

\renewcommand{\algorithmiccomment}[1]{\bgroup\hfill //~#1\egroup}
\usepackage{enumitem}
\usepackage{graphics}
\usepackage{multirow}
\usepackage{booktabs}
\usepackage{graphicx}
\usepackage{caption}
\usepackage{subcaption}

\usepackage{algorithm,algorithmic,refcount}